\documentclass[microtype]{gtart_a}
\pdfoutput=1
\title{$\mathrm{G}_2$--instantons on generalised Kummer constructions}

\author{Thomas Walpuski}
\givenname{Thomas}
\surname{Walpuski}
\email{tw809@ic.ac.uk}
\address{Department of Mathematics\\
Imperial College London\\\newline
180 Queen's Gate\\
London SW7 2AZ\\
UK}

\volumenumber{17}
\issuenumber{4}
\publicationyear{2013}
\papernumber{52}
\startpage{2345}
\endpage{2388}


\received{26 January 2012}
\revised{17 December 2012}
\accepted{2 May 2013}
\published{22 August 2013}
\publishedonline{22 August 2013}
\proposed{Ronald Stern}
\seconded{Richard Thomas, Simon Donaldson}
\editor{NJ}
\version{3}

\arxivreference{1109.6609}

\makeatletter
\def\cnewtheorem#1[#2]#3{\newtheorem{#1}{#3}[section]
\expandafter\let\csname c@#1\endcsname\c@thm}
\makeatother

\let\xysavmatrix\xymatrix
\def\xymatrix{\disablesubscriptcorrection\xysavmatrix}
\AtBeginDocument{\let\bar\wbar\let\tilde\wtilde\let\hat\what}

\newcommand{\vstrut}[2]{{\vrule height #1pt depth #2pt width 0pt}} 

\newcommand{\tC}[1]{\widetilde{\vstrut{8}{0}\smash{\C^2/#1}}}

\numberwithin{equation}{section}

\theoremstyle{plain}
\newtheorem{theorem}{Theorem}[section]
\newtheorem{prop}[theorem]{Proposition}

\newtheorem{lemma}[theorem]{Lemma}

\newtheorem{step}{Step}
\newtheorem{case}{Case}

\theoremstyle{remark}
\newtheorem{remark}[theorem]{Remark}

\theoremstyle{definition}
\newtheorem{definition}[theorem]{Definition}
\newtheorem{example}[theorem]{Example}

\newcommand{\N}{\mathbb{N}}
\newcommand{\HH}{\mathbb{H}}
\newcommand{\Aff}{{\rm Aff}}
\newcommand{\GL}{{\rm GL}}
\newcommand{\SO}{{\rm SO}}
\newcommand{\U}{{\rm U}}
\newcommand{\SU}{{\rm SU}}
\newcommand{\Spin}{{\rm Spin}}
\newcommand{\Diff}{\mathrm{Diff}}
\newcommand{\so}{\mathfrak{so}}
\newcommand{\rd}{{\rm d}}
\newcommand{\rG}{{\rm G}}
\newcommand{\sA}{\mathscr A}
\newcommand{\sG}{\mathscr G}
\newcommand{\sM}{\mathscr M}
\newcommand{\sP}{\mathscr P}
\newcommand{\sS}{\mathscr S}
\newcommand{\cH}{\mathcal H}
\newcommand{\cL}{\mathcal L}
\newcommand{\cR}{\mathcal R}
\newcommand{\frg}{\mathfrak g}
\newcommand{\bg}{{\bf g}}
\newcommand{\br}{{\bf r}}
\newcommand{\into}{\hookrightarrow}
\newcommand{\del}{\partial}
\newcommand{\iso}{\cong}
\newcommand{\Aut}{\mathrm{Aut}}
\newcommand{\ch}{{\rm ch}}
\newcommand{\id}{\mathrm{id}}

\newcommand{\tr}{\mathop{\mathrm{tr}}\nolimits}
\newcommand{\vol}{\mathrm{vol}}
\newcommand{\PD}{{\rm PD}} 
\newcommand{\Ric}{{\rm Ric}} 
\newcommand{\loc}{{\rm loc}}

\def\({\mathopen{}\left(}
\def\){\right)\mathclose{}}

\newcommand{\Hol}{{\rm Hol}} 
\newcommand{\gl}{\mathfrak{gl}}
\newcommand{\ua}{{\underline a}}
\newcommand{\ub}{{\underline b}}

\newcommand{\dvol}{{\rm dvol}}
\newcommand{\YM}{{\rm YM}}
\newcommand{\diag}{{\rm diag}}

\newcommand{\qandq}{\quad\text{and}\quad}
\def\<{\langle}
\def\>{\rangle}

\begin{document}

\begin{asciiabstract}
In this article we introduce a method to construct G_2-instantons on
G_2-manifolds arising from Joyce's generalised Kummer construction. The
method is based on gluing ASD instantons over ALE spaces to flat bundles
on G_2-orbifolds of the form T^7/Gamma.  We use this construction to
produce non-trivial examples of G_2-instantons.
\end{asciiabstract}

\begin{webabstract}
In this article we introduce a method to construct
$\mathrm{G}_2$--instantons on $\mathrm{G}_2$--manifolds arising
from Joyce's generalised Kummer construction. The method
is based on gluing ASD instantons over ALE spaces to flat
bundles on $\mathrm{G}_2$--orbifolds of the form $T^7/ \Gamma$.
We use this construction to produce non-trivial examples of
$\mathrm{G}_2$--instantons.
\end{webabstract}

\begin{abstract}
In this article we introduce a method to construct
$\mathrm{G}_2$--instantons on $\mathrm{G}_2$--manifolds arising from Joyce's
generalised Kummer construction~\cite{Joyce1996a,Joyce1996b}.  The method is
based on gluing ASD instantons over ALE spaces to flat bundles on
$\mathrm{G}_2$--orbifolds of the form $T^7\!/\Gamma$.  We use this
construction to produce non-trivial examples of $\mathrm{G}_2$--instantons.
\end{abstract}

\maketitle

\section{Introduction}
\label{sec:intro}

The seminal paper~\cite{Donaldson1998} of Donaldson--Thomas has
inspired a considerable amount of work related to gauge theory in
higher dimensions.  Tian~\cite{Tian2000} and Tao--Tian~\cite{Tao2004}
made significant progress on important foundational analytical
questions.  Recent work of Donaldson--Segal~\cite{Donaldson2009} and
Haydys~\cite{Haydys2011} shed some light on the shape of the theories
to be expected.

In this article we will focus on the study of gauge theory on
$\rG_2$--manifolds.  These are $7$--manifolds equipped with a
torsion-free $\rG_2$--structure.  The $\rG_2$--structure allows us to
define a special class of connections, called $\rG_2$--instantons (see
\fullref{def:i}).  These share many formal properties with flat
connections on $3$--manifolds and it is expected that there are
$\rG_2$--analogues of those $3$--manifold invariants that are related
to ``counting flat connections'', that is, the Casson invariant,
instanton Floer homology, etc.

So far non-trivial examples of $\rG_2$--instantons are rather rare.
By exploiting the special geometry of the known $\rG_2$--manifolds
some progress has been made recently.  At the time of writing, there
are essentially two methods for constructing compact
$\rG_2$--manifolds in the literature.  Both yield $\rG_2$--manifolds
close to degenerate limits.  One is Kovalev's twisted connected sum
construction~\cite{Kovalev2003}, which produces $\rG_2$--manifolds
with ``long necks'' from certain pairs of Calabi--Yau 3--folds with
asymptotically cylindrical ends.  A technique for constructing
$\rG_2$--instantons on Kovalev's $\rG_2$--manifolds has recently been
proposed by S\'a~Earp~\cite{SaEarp2011,SaEarp2011a}.  The other (and
historically the first) method for constructing $\rG_2$--manifolds is
due to Joyce~\cite{Joyce1996a,Joyce1996b} and is based on desingularising $\rG_2$--orbifolds.  In this article we introduce a method to construct
$\rG_2$--instantons on $\rG_2$--manifolds arising from Joyce's
construction.

To set up the framework for our construction, let us briefly review
the geometry of Joyce's construction: Equip $T^7$ with a flat
$\rG_2$--structure $\phi_0$ and let $\Gamma$ be a finite group of
diffeomorphisms of $T^7$ preserving $\phi_0$.  Then
$Y_0:={T^7}/\Gamma$ is a flat $\rG_2$--orbifold.  The singular set $S$
of $Y_0$ can, in general, be quite complicated.  In this article we
restrict to \emph{admissible} $\rG_2$--orbifolds $Y_0$.  That is, we
assume that each of the connected components $S_j$ of $S$ has a
neighbourhood modelled on $(T^3\times \C^2/G_j)/H_j$.  Here $G_j$ is
a non-trivial finite subgroup of $\SU(2)$ and $H_j$ is a finite group
acting by isometries on $T^3$ as well as on $\C^2/G_j$; moreover, the
action of $H_j$ on $T^3=\R^3/\Z^3$ is induced by a free affine action
on $\R^3$ normalising the action of $\Z^3$.  Suppose we are given
\emph{resolution data} $\br=\{(X_j,\rho_j)\}$ for $Y_0$, that is, for
each $j$, an ALE space $X_j$ asymptotic to $\C^2/G_j$ together with an
isometric action $\rho_j$ of $H_j$ on $X_j$ which is asymptotic to the
action of $H_j$ on $\C^2/G_j$.  Then using Joyce's generalised Kummer
construction~\cite{Joyce1996a,Joyce1996b} we can resolve the singularities in
$Y_0$ and produce a compact $7$--manifold $Y$ together with a family
of torsion-free $\rG_2$--structures $\(\phi_t\)_{t\in(0,T)}$.

In this article we will construct $\rG_2$--instantons over
$(Y,\phi_t)$ given \emph{gluing data} $\bg$ compatible with the
resolution data $\br$ for $Y_0$.  The notion of gluing data will be
defined carefully in \fullref{sec:approx}.  For now, it suffices
to say that $\bg$ consists of
\begin{itemize}
\item a $G$--bundle $E_0$ over $Y_0$ together with a flat connection
  $\theta$ and
\item for each $j$, a $G$--bundle $E_j$ over $X_j$ together with a
  framed ASD instanton $A_j$
\end{itemize}
as well as various auxiliary data satisfying a number of compatibility
conditions.  Here we take $G$ to be a compact connected semi-simple
Lie group, for example, $G=\SO(3)$.

\begin{theorem}\label{thm:a}
  Let $Y_0$ be an admissible flat $\rG_2$--orbifold, let $\br$ be
  resolution data for $Y_0$ and let $\bg$ be compatible gluing data.
  Suppose that the flat connection $\theta$ is acyclic and that the
  ASD instantons $A_j$ are infinitesimally rigid.  Then there is a
  constant $T'\in(0,T]$ and a $G$--bundle $E$ over $Y$ as well as for
  each $t\in(0,T')$ a connection $A_t$ on $E$ that is an acyclic
  $\rG_2$--instanton over $(Y,\phi_t)$.  Moreover, the adjoint bundle
  $\frg_E$ associated with $E$ satisfies
  \begin{align}\label{eq:p1}
    p_1(\frg_{E})&=-\sum_j k_j\,\PD[S_j]\quad\text{with }
    k_j:=\frac{1}{8\pi^2}\int_{X_j} |F_{A_j}|^2, \\
  \label{eq:w2}
\text{and}\quad
    \<w_2(\frg_{E}),\Sigma\>&=\<w_2(\frg_{E_j}),\Sigma\>
  \end{align}
  for each $\Sigma\in H_2(X_j)^{H_j}\subset H_2(Y)$.  Here $[S_j]\in
  H_3(Y,\Q)$ is the rational homology class arising from $S_j$ and
  $H_2(X_j)^{H_j}$ denotes the $H_j$--invariant part of $H_2(X_j)$;
  see \fullref{rmk:topy}.
\end{theorem}

\begin{remark}
  We will specify in \fullref{def:acyclic} and
  \fullref{def:rigid}, respectively, what it means for a
  $\rG_2$--instanton, and thus for a flat connection, being a
  particular instance of a $\rG_2$--instanton, to be acyclic and
  for an ASD instanton to be infinitesimally rigid.
\end{remark}

\begin{remark}
  We equip the adjoint bundles $\frg_{E_j}$ and $\frg_E$ with the
  inner product arising from the negative of the Killing form on the
  Lie algebra $\frg$ associated with $G$.
\end{remark}

It is not unreasonable to expect that under certain topological
assumptions all $\rG_2$--instantons on $\rG_2$--manifolds arsing from
Joyce's generalised Kummer construction close to the degenerate limit
come from a suitable generalisation of our construction.
Optimistically, one could hope that this will some day make the (so
far conjectural) $\rG_2$ Casson invariant accessible to computation.

The proof of \fullref{thm:a} is based on a gluing construction.
The analysis involved is similar to work on $\Spin(7)$--instantons in
Lewis' DPhil thesis~\cite{Lewis1998}, unpublished work of Brendle on the
Yang--Mills equation in higher dimension~\cite{Brendle2003} and
Pacard--Ritor\'e's work on the Allen--Cahn equation~\cite{Pacard2003}.
From a geometric perspective our result can be viewed as a
higher-dimensional analogue of Kronheimer's work on ASD instantons on
Kummer surfaces~\cite{Kronheimer1991}.

Here is an outline of the article.  Sections~\ref{sec:review},
\ref{sec:g2i}, \ref{sec:kummer} and \ref{sec:asdale} contain some
foundational material on $\rG_2$--manifolds and $\rG_2$--instantons as
well as brief reviews of Joyce's generalised Kummer construction and
Kronheimer and Nakajima's work on ASD instantons on ALE spaces.  The
proof of \fullref{thm:a} begins in earnest in
\fullref{sec:approx}, where we construct approximate
$\rG_2$--instantons from gluing data and introduce weighted H\"older
spaces adapted to the problem at hand.  In \fullref{sec:model} we
set up the analytical problem underlying the proof of
\fullref{thm:a} and discuss a model for the linearised problem.
We complete the proof of \fullref{thm:a} in
\fullref{sec:deform}.  A number of concrete examples of
$\rG_2$--instantons with $G=\SO(3)$ are constructed in
\fullref{sec:ex}.

\subsubsection*{Acknowledgements}
This article is the outcome of work undertaken by the author for his
PhD thesis at Imperial College London, supported by European Research
Council Grant 247331.  I am grateful to my supervisor Simon Donaldson
his encouragement and support.  Moreover, I would like to thank the
anonymous referee for helpful comments on an earlier version of this
article.

\section[Review of G_2--manifolds]{Review of $\rG_2$--manifolds}
\label{sec:review}

In this section we recall some basic definitions and results in
$\rG_2$--geometry.  For a more comprehensive treatment we refer the
reader to Joyce's book~\cite{Joyce2000}, specifically Chapter~10.

The Lie group $\rG_2$ can be defined as the subgroup of elements of
$\GL(7)$ fixing the $3$--form
\begin{equation}\label{eq:phi0}
  \phi_0:=\rd x^{123}+\rd x^{145}+\rd x^{167}+\rd x^{246}-\rd
  x^{257}-\rd x^{347}-\rd x^{356}.
\end{equation}
Here $\rd x^{ijk}$ is a shorthand for $\rd x^i\wedge \rd x^j \wedge
\rd x^k$ and $x_1,\ldots,x_7$ are standard coordinates on $\R^7$.  The
particular choice of $\phi_0$ is not important.  Any non-degenerate
$3$--form $\phi$ on $\R^7$ is equivalent to $\phi_0$ under a change of
coordinates; see, for example, Salamon--Walpuski~\cite[Theorem~3.2]{Salamon2010}.  Here we say
that $\phi$ is non-degenerate if for each non-zero vector $u\in\R^7$
the $2$--form $i(u)\phi$ on $\R^7/\<u\>$ is symplectic.  It follows
from the identity%
\begin{equation}\label{eq:gphi}
  i(u)\phi_0\wedge i(v)\phi_0\wedge\phi_0=6g_{\R^7}(u,v) \vol_{\R^7}
\end{equation}
that any element of $\GL(7)$ which preserves $\phi_0$ also preserves
the standard inner product $g_{\R^7}$ and the standard volume form
$\vol_{\R^7}$ on $\R^7$.  Therefore, $\rG_2$ is a subgroup of
$\SO(7)$. In particular, every non-degenerate $3$--form $\phi$ on a
$7$--dimensional vector space induces an inner product and an
orientation on this vector space.  As an aside, we should point out
here that non-degenerate $3$--forms constitute one of two open orbits
of $\GL(7)$ in $\Lambda^3 (\R^7)^*$.  For $\phi$ in the other open
orbit, the analogue of equation~\eqref{eq:gphi} yields an indefinite
metric of signature $(3,4)$.  In particular, if we take $u=v$ to be a
light-like vector, then $i(u)\phi$ is not a symplectic form on
$\R^7/\<u\>$.

From the above discussion it is clear that a non-degenerate $3$--form
$\phi$ on $Y$ is equivalent to a reduction of the structure group of
$TY$ from $\GL(7)$ to $\rG_2$, that is, a $\rG_2$--structure.  Moreover,
$\phi$ induces a Riemannian metric $g_\phi$ and an orientation on $Y$.
The intrinsic torsion of the $\rG_2$--structure corresponding to
$\phi$ can be identified with $\nabla_{g_\phi} \phi$.

\begin{definition}
  A \emph{$\rG_2$--manifold} is a $7$--manifold $Y$ equipped with a
  torsion-free $\rG_2$--structure $\phi$, that is,
  \begin{equation*}
    \nabla_{g_\phi} \phi=0.
  \end{equation*}
\end{definition}

\begin{remark}
  Analogously, one can define the general notion of a
  \emph{$\rG_2$--orbifold}.  (For a thorough discussion of orbifolds
  we recommend the book of Adem--Leida--Ruan~\cite{Adem2007}.)  In
  this article, however, we will only encounter very simple
  $\rG_2$--orbifolds of the form $(Y/\Gamma,\phi)$ where $(Y,\phi)$ is
  a $\rG_2$--manifold and $\Gamma$ is a finite group of diffeomorphism
  of $Y$ preserving $\phi$.
\end{remark}

There is a plethora of reasons to be interested in $\rG_2$--manifolds.
$\rG_2$--manifold have holonomy group $\Hol(g_\phi)\subset\rG_2$ which
appears as one of the exceptional cases in Berger's classification of
holonomy groups of irreducible non-symmetric Riemannian
manifolds~\cite[Theorem~3]{Berger1955}.  $\rG_2$--manifolds are spin manifolds
and carry (at least) one non-zero parallel spinor (see
Joyce~\cite[Proposition~10.1.6]{Joyce2000}) and, hence, are Ricci-flat and of
relevance to theoretical physics.  Moreover, $\rG_2$--manifolds carry
a pair of calibrations in the sense of Harvey--Lawson~\cite{Harvey1982}: the \emph{associative calibration} $\phi$ and the
\emph{coassociative calibration} $\psi:=*\phi$.  This makes their
submanifold geometry very rich and interesting.  Furthermore, it is
very appealing to study gauge theory on $\rG_2$--manifolds as we will
see in \fullref{sec:g2i}.

\begin{example}\label{ex:t7}
  The $7$--torus $T^7=\R^7/\Z^7$ equipped with the $\rG_2$--structure
  $\phi_0$ defined in \eqref{eq:phi0} is a $\rG_2$--manifold.
\end{example}

\begin{definition}
  A \emph{hyperk\"ahler manifold} is a Riemannian manifold $(X,g)$
  together with a triple $(I_1,I_2,I_3)$ of parallel orthogonal
  complex structures satisfying $I_1I_2=-I_2I_1=I_3$.
\end{definition}

\begin{remark}
  If $(X,g,I_1,I_2,I_3)$ is a hyperk\"ahler manifold, then the metric
  $g$ is K\"ahler with respect to each of complex structures
  $a_1I_1+a_2I_2+a_3I_3$ with $(a_1,a_2,a_3)\in S^2\subset \R^3$.
\end{remark}

\begin{example}\label{ex:t3x}
  Let $(X,g,I_1,I_2,I_3)$ be a hyperk\"ahler $4$--manifold.  For
  $i=1,2,3$ denote by $\omega_i:=g(I_i\,\cdot\,,\,\cdot\,)$ the K\"ahler form
  associated with the complex structure $I_i$.  Choose an orthonormal
  triple $(\delta^1,\delta^2,\delta^3)$ of constant $1$--forms on
  $T^3$.  Then $T^3\times X$ is a $\rG_2$--manifold with torsion-free
  $\rG_2$--structure $\phi$ defined by
  \begin{equation*}
    \phi:=\delta^1\wedge\delta^2\wedge\delta^3
          +\delta^1\wedge\omega_1
          +\delta^2\wedge\omega_2
          -\delta^3\wedge\omega_3.
  \end{equation*}
  The metric and the orientation on $T^3\times X$ induced by $\phi$
  coincide with the product metric and the product orientation.  To
  see that, note that each cotangent space to $X$ has a positive
  orthonormal basis $(e^0,\ldots,e^3)$ with $e^i=I_ie^0$, for
  $i=1,2,3$, such that%
  \begin{equation}\label{eq:omega}
\begin{aligned}
    \omega_1&=e^0\wedge e^1+e^2\wedge e^3, \\
    \omega_2&=e^0\wedge e^2-e^1\wedge e^3, \\
    \omega_3&=e^0\wedge e^3+e^1\wedge e^2.
\end{aligned}
  \end{equation}
  This immediately yields a orientation-preserving isometry
  $T_x(T^3\times X) \to \R^7$ identifying $\phi$ with $\phi_0$.
  Note that in the current example the coassociative calibration
  $\psi:=*\phi$ is given by
  \begin{equation}\label{eq:psi}
    \psi=\tfrac12 \omega_1\wedge\omega_1
    + \delta^2\wedge\delta^3\wedge\omega_1
    + \delta^3\wedge\delta^1\wedge\omega_2
    - \delta^1\wedge\delta^2\wedge\omega_3.
  \end{equation}
\end{example}

\begin{remark}\label{rmk:holpi1}
  The above examples have holonomy strictly contained in $\rG_2$.
  This is clear from their construction, but can also be seen as a
  consequence of their topology since a compact $\rG_2$--manifold
  $(Y,\phi)$ satisfies $\Hol(g_\phi)=\rG_2$ if and only if $\pi_1(Y)$
  is finite; see Joyce~\cite[Proposition~10.2.2]{Joyce2000}.
\end{remark}

The following observation is central for the construction of
$\rG_2$--manifolds.

\begin{theorem}[Fern\'andez--Gray~{\cite[Theorem~4.9]{Fernandez1982}}]\label{thm:fg}
  Let $Y$ be a $7$--manifold.  Denote by
  $\sP\subset\Omega^3(Y)$ the subspace of all
  non-degenerate $3$--forms on $Y$ and define
  $\Theta\co \sP \to \Omega^4(Y)$ by
  \begin{equation}\label{eq:Theta}
    \Theta(\phi):=*_\phi \phi.
  \end{equation}
  Here $*_\phi$ is the Hodge $*$--operator associated with $\phi$.
  Then a $\rG_2$--structure $\phi$ is torsion-free if and only if
  \begin{equation*}
    \rd\phi=0 \quad\text{and}\quad \rd\Theta(\phi)=0.
  \end{equation*}
\end{theorem}

The key difficulty in constructing $\rG_2$--manifolds comes from the
fact that $\Theta$ is non-linear.  It is currently unknown which
compact $7$--manifolds do admit torsion-free $\rG_2$--structures.  All
known non-trivial compact examples arise by way of gluing
constructions.  One of those constructions will be described in more
detail in \fullref{sec:kummer}.

Before we move on, let us recall a few facts, going back at least to
the work of Fern\'andez--Gray~\cite{Fernandez1982}, that will be useful
in the following.  We refer the interested reader to
Salamon--Walpuski~\cite[Theorem~8.4]{Salamon2010} for a detailed proof.

\begin{prop}\label{prop:split}
  There is a $\rG_2$--invariant orthogonal splitting
  \begin{equation*}
    \Lambda^2 (\R^7)^*=\Lambda^2_7\oplus\Lambda^2_{14},
  \end{equation*}
  where
  \begin{equation*}
    \Lambda^2_7
      := \left\{ \omega : *(\omega\wedge\phi_0)=2\omega\right\}
    \quad\text{and}\quad
    \Lambda^2_{14}
      := \left\{ \omega : *(\omega\wedge\phi_0)=-\omega\right\}.
  \end{equation*}
  Moreover, $\Lambda^2_{14}$ is the kernel of the map $\omega\mapsto
  \omega\wedge\psi_0$, where $\psi_0:=*\phi_0$, and can be identified
  with $\frg_2\subset \so(7)\iso\Lambda^2(\R^7)^*$.
\end{prop}

\section[Gauge theory on G_2--manifolds]
{Gauge theory on $\rG\sb2$--manifolds}
\label{sec:g2i}

Let $(Y,\phi)$ be a \emph{compact} $\rG_2$--manifold (or, more
generally, a compact $\rG_2$--orbifold), let $\psi:=\Theta(\phi)$ and
let $E$ be a $G$--bundle over $Y$.  Denote by $\sA(E)$ the space of
connections on $E$.

\begin{definition}\label{def:i}
  A connection $A\in\sA(E)$ on $E$ is called a
  \emph{$\rG_2$--instanton} if it satisfies
  \begin{equation}\label{eq:i1}
    *(F_A\wedge\phi)=-F_A.
    \vadjust{\penalty-2000}
  \end{equation}
\end{definition}

These equations have first appeared in the physics literature
(see Corrigan--Devchand--Fairlie--Nuyts~\cite{Corrigan1983}) and were later brought to a wider attention by
Donaldson--Thomas~\cite[Section~3]{Donaldson1998}.
\fullref{eq:i1} can be thought of as a $7$--dimensional version
of the anti-self-duality condition familiar from dimension four.  As
we will discuss shortly, $\rG_2$--instantons also have a striking
similarity with flat connections over $3$--manifolds.%

\begin{example}\label{ex:flat}
  Flat connections are $\rG_2$--instantons.
\end{example}

\begin{example}\label{ex:asdg2}
  Let $X$ be a hyperk\"ahler manifold, let $E$ be a $G$--bundle over
  $X$ and let $A$ be an ASD instanton on $E$, that is, a connection on
  $E$ whose curvature $F_A$ is anti-self-dual.  Then the pullback of
  $A$ to the $\rG_2$--manifold $T^3\times X$ from \fullref{ex:t3x}
  is a $\rG_2$--instanton:
  \begin{equation*}
    *(F_A\wedge\phi)
    =*\big(F_A\wedge\delta^1\wedge\delta^2\wedge\delta^3\big)
    =*_X F_A = -F_A.
  \end{equation*}
  Here we used that $F_A\wedge\omega_i=0$ and $*_X$ denotes the
  Hodge $*$--operator on $X$.
\end{example}

\begin{example}
  The Levi-Civita connection on a $\rG_2$--manifold $(Y,\phi)$ is a
  $\rG_2$--instanton.  To see that, observe that at each point we can
  think of the Riemannian curvature tensor $R$ as an element of
  $S^2\frg_2\subset \Lambda^2\otimes\gl(7)$, since
  $\Hol(g_\phi)\subset\rG_2$.  But then it follows from
  \fullref{prop:split} that $*(R\wedge\phi)=-R$.
\end{example}

Since $\phi$ is closed, it follows from the Bianchi identity that
$\rG_2$--instantons are Yang--Mills connections, that is,
$\rd_A^*F_A=0$.  In fact, they are absolute minima of the Yang--Mills
functional $\mathrm{YM}\co \sA(E) \to \R$, since
\begin{equation}\label{eq:ei}
  \YM(A)
  :=\int_Y |F_A|^2 \dvol
   = \tfrac13 \int_Y |F_A+*(F_A\wedge\phi)|^2 \dvol
   - \int_Y \<F_A\wedge F_A\>\wedge\phi
\end{equation}
and, by Chern--Weil theory, the second term is a topological constant
depending only on $E$.  The energy identity \eqref{eq:ei} follows from
a straight-forward computation using \fullref{prop:split}.

\begin{prop}\label{prop:ied}
  Let $A\in\sA(E)$ be a connection on $E$. The following are equivalent.
  \begin{enumerate}
  \item\label{ied:1} $A$ is $\rG_2$--instanton.
  \item\label{ied:2} $A$ satisfies $F_A\wedge\psi=0$.
  \item\label{ied:3} There is a $\xi\in\Omega^0(Y,\frg_E)$ such that
    \begin{equation}\label{eq:i2}
      *(F_A\wedge\psi)+\rd_A\xi=0.
      \vadjust{\goodbreak}
    \end{equation}
  \end{enumerate}
\end{prop}

\begin{proof}
  The equivalence of \eqref{ied:1} and \eqref{ied:2} follows
  immediately from \fullref{prop:split}.  Obviously,
  \eqref{ied:2} implies \eqref{ied:3}.  By the Bianchi identity and
  since $\rd\psi=0$ it follows from \eqref{ied:3} that $\rd_A^*\rd_A\xi=0$.
  Hence, by integration by parts,
  \begin{equation*}
    \int_Y |\rd_A\xi|^2=\int_Y \<\rd_A^*\rd_A\xi,\xi\>=0.
  \end{equation*}
  Therefore $\rd_A\xi=0$ and \eqref{ied:3} implies \eqref{ied:2}.
\end{proof}

From \fullref{prop:ied} it becomes apparent that
$\rG_2$--instantons are rather similar to flat connections on
$3$--manifolds.  In particular, if $A_0$ is a $\rG_2$--instanton on
$E$, then there is a $\rG_2$ Chern--Simons functional
$CS^\psi\co\sA(E)\to\R$ defined by
\begin{equation*}
  CS^\psi(A_0+a):=\int_Y \bigl\langle a\wedge\rd_{A_0} a + \tfrac13
  a\wedge[a\wedge a]@\bigr\rangle\wedge\psi
\end{equation*}
whose critical points are precisely the $\rG_2$--instantons on $E$.
It is not entirely unreasonable to expect that some of the
$3$--manifold invariants arising from the Chern--Simons functional,
like the Casson invariant and instanton Floer homology, have
$\rG_2$--analogues.  This idea goes back at least to the seminal paper
of Donaldson--Thomas~\cite{Donaldson1998} and is one of the main
motivations for studying $\rG_2$--instantons.  
Since \fullref{eq:i1} is
invariant under the action of the group $\sG$ of gauge transformations
of $E$, we can consider the \emph{moduli space} of $\rG_2$--instantons
on $E$ over $(Y,\phi)$:
\begin{equation*}
  \sM(E,\phi):=\left\{ A\in\sA(E) : F_A\wedge\psi=0 \right\}/\sG.
\end{equation*}
Very roughly speaking, the conjectural \emph{$\rG_2$ Casson invariant}
should be obtained by ``counting'' $\sM(E,\phi)$.  Whether there is a
rigorous construction of such a $\rG_2$ Casson invariant and whether
it can, in fact, be arranged to be invariant under isotopies of the
$\rG_2$--structure is an open question.  A brief discussion of parts
of this circle of ideas can be found in Donaldson--Segal~\cite[Section~6]{Donaldson2009}.

It is customary in gauge theory to work with local slices of the gauge
group action.  A particularly useful slicing condition is to require
that $B\in\sA(E)$ be in \emph{Coulomb gauge} with respect to a fixed
reference connection $A\in\sA(E)$, that is, $\rd_A^*(B-A)=0$.  (The
importance of the Coulomb gauge stems from the foundational work of
Uhlenbeck~\cite{Uhlenbeck1982a}.  For a careful discussion of how the
Coulomb gauge is used in the construction moduli spaces we refer the
reader to Donaldson--Kronheimer~\cite[Section~4.2]{Donaldson1990}.)  For a fixed connection
$A\in\sA(E)$ we consider the system of equations 
\begin{equation}\label{eq:i2g}
  *(F_{A+a}\wedge\psi)+\rd_{A+a}\xi=0 \quad\text{and}\quad
  \rd_A^*a=0
  \vadjust{\goodbreak}
\end{equation}
for $\xi\in\Omega^0(Y,\frg_E)$ and $a\in\Omega^1(Y,\frg_E)$.  This is
simply \eqref{eq:i2} for $A+a$ instead of $A$ together with the
condition that $A+a$ be in Coulomb gauge with respect to $A$.  The
linearisation $L_A\co\Omega^0(Y,\frg_E)\oplus\Omega^1(Y,\frg_E) \to
\Omega^0(Y,\frg_E)\oplus\Omega^1(Y,\frg_E)$ of \eqref{eq:i2g} is given
by
\begin{equation}\label{eq:la}
  L_A:=\begin{pmatrix}
    0 & \rd_A^* \\
    \rd_A & *\left(\psi\wedge\rd_A\right)
  \end{pmatrix}.
\end{equation}
This is a self-adjoint elliptic operator.  If $A\in\sA(E)$ is a
$\rG_2$--instanton, then $L_A$ controls the infinitesimal deformation
theory of $A$ as a $\rG_2$--instanton.

\begin{definition}\label{def:acyclic}
  A $\rG_2$--instanton $A$ is called \emph{acyclic} if the
  operator $L_A$ is invertible.
\end{definition}

One can show that if every $\rG_2$--instanton $A$ on $E$ is
acyclic, then $\sM(E,\phi)$ is, in fact, a smooth
zero-dimensional manifold, that is, a discrete set.

\section{Joyce's generalised Kummer construction}
\label{sec:kummer}

Equip $T^7$ with a flat $\rG_2$--structure $\phi_0$, as in
\fullref{ex:t7}, and let $\Gamma$ be a finite group of
diffeomorphisms of $T^7$ preserving $\phi_0$.  Then $Y_0:=T^7/\Gamma$
is a flat $\rG_2$--orbifold.  Denote by $S$ the singular set of $Y_0$
and denote by $S_1, \ldots, S_k$ its connected components.

\begin{definition}
  $Y_0$ is called \emph{admissible} if each $S_j$ has a neighbourhood
  isometric to a neighbourhood of the singular set of 
$(T^3\times
  \C^2/G_j)/H_j$.  Here $G_j$ is a non-trivial finite subgroup of
  $\SU(2)$ and $H_j$ is a finite group acting by isometries on $T^3$
  as well as on $\C^2/G_j$; moreover, the action of $H_j$ on
  $T^3=\R^3/\Z^3$ is induced by a free affine action on $\R^3$
  normalising the action of $\Z^3$.
\end{definition}

Let $Y_0$ be an admissible flat $\rG_2$--orbifold.  Then there is a
constant $\zeta>0$ such that if we denote by $T$ the set of points at
distance less that $\zeta$ to $S$, then $T$ decomposes into
connected components $T_1, \ldots, T_k$ such that $T_j$ contains $S_j$
and is isometric to $(T^3\times B^4_\zeta/G_j)/H_j$.  On $T_j$ we can
write
\begin{equation*}
  \phi_0=\delta^1\wedge\delta^2\wedge\delta^3
        +\delta^1\wedge\omega_1
        +\delta^2\wedge\omega_2
        -\delta^3\wedge\omega_3,
\end{equation*}
where $(\delta^1,\delta^2,\delta^3)$ is an orthonormal triple of
constant $1$--forms on $T^3$ and where $(\omega_1,\omega_2,\omega_3)$ is the
triple of K\"ahler forms associated with the standard hyperk\"ahler
structure $(g,I_1,I_2,I_3)$ on $\C^2\iso\HH$.
\vadjust{\goodbreak}

\begin{definition}\label{def:ale}
  Let $G$ be a finite subgroup of $\SU(2)$.  Then an \emph{ALE space
    asymptotic to $\C^2/G$} is a hyperk\"ahler $4$--manifold $(X,\hat
  g,\hat I_1,\hat I_2,\hat I_3)$ together with a continuous map
  $\pi\co X\to\C^2/G$ inducing a diffeomorphism from
  $X\setminus\pi^{-1}(0)$ to $(\C^2\setminus\{0\})/G$ such that
  \begin{equation}\label{eq:aledecay}
    \nabla^k(\pi_*\hat g-g)=O\big(r^{-4-k}\big) \quad\text{and}\quad
    \nabla^k(\pi_*\hat I_i-I_i)=O\big(r^{-4-k}\big)
  \end{equation}
  as $r\to\infty$ for $i=1,2,3$ and $k\geq 0$.  Here $r\co
  \C^2/G_j\to[0,\infty)$ denotes the radius function.
\end{definition}

We will remove the singularity in $Y_0$ along $S_j$ by, roughly speaking,
replacing each $\C^2/G_j$ with an ALE space asymptotic to $\C^2/G_j$.
Due to work of Kronheimer~\cite{Kronheimer1989,Kronheimer1989a}, ALE
spaces are very well understood.

\begin{theorem}[Kronheimer~{\cite[Theorems~1.1,~1.2
    and~1.3]{Kronheimer1989a}}]\label{thm:ale}
  Let $G$ be a non-trivial finite subgroup of $\SU(2)$.  Denote by $X$
  the real $4$--manifold underlying the crepant resolution
  $\tC{G}$.  Then for each three cohomology classes
  $\alpha_1,\alpha_2,\alpha_3\in H^2(X,\R)$ satisfying
  \begin{equation}\label{eq:ndale}
    (\alpha_1(\Sigma),\alpha_2(\Sigma),\alpha_3(\Sigma))\neq 0 \in \R^3
  \end{equation}
  for each $\Sigma\in H_2(X,\Z)$ with $\Sigma\cdot\Sigma=-2$ there is
  a unique ALE hyperk\"ahler structure on $X$ for which the cohomology
  classes of the K\"ahler forms $[\omega_i]$ are given by $\alpha_i$.
  Moreover, each ALE space asymptotic to $\C^2/G$ is diffeomorphic to
  $\tC{G}$ and its associated triple of K\"ahler classes
  satisfies \eqref{eq:ndale}.
\end{theorem}

\begin{remark}\label{rmk:mckay}
  The crepant resolution $\tC{G}$ can be obtained from
  $\C^2/G$ by a sequence of blow-ups.  The exceptional divisor $E$ of
  $X=\smash{\tC{G}}$ has irreducible components
  $\Sigma_1,\ldots,\Sigma_k$.  By the McKay correspondence~\cite{McKay1980}, these components form a basis of $H_2(X,\Z)$ and
  the matrix with coefficients $C_{ij}=-[\Sigma_i]\cdot[\Sigma_j]$ is
  the Cartan matrix associated with the Dynkin diagram corresponding
  to $G$ in the ADE classification of finite subgroups of $\SU(2)$.
\end{remark}

\begin{definition}
  A collection $\br=\{(X_j,\rho_j)\}$ consisting of, for each $j$, an
  ALE space $X_j$ asymptotic to $\C^2/G_j$ together with an isometric
  action $\rho_j$ of $H_j$ on $X_j$ which is asymptotic to the action
  of $H_j$ on $\C^2/G_j$ is called \emph{resolution data for $Y_0$}.
\end{definition}

Suppose we are given resolution data $\br=\{(X_j,\rho_j)\}$.  Denote
by $\pi_j\co X_j\to\C^2/G_j$ the resolution map for $X_j$. For $t>0$
define
\begin{equation}\label{eq:pijt}
  \pi_{j,t}:=t\pi_j:X_j\to\C^2/G_j
\end{equation}
and set
\begin{equation}\label{eq:Tjt}
  \tilde T_{j,t}:=\big(T^3\times
\pi_{j,t}^{-1}\big(B^4_{\zeta}/G_j\big)\big)/H_j
  \quad\text{and}\quad
  \tilde T_t:=\bigcup_j \tilde T_{j,t}.
\end{equation}
Using $\pi_{j,t}$ we can replace each $T_j$ in $Y_0$ by $\tilde
T_{j,t}$ and thus obtain a compact $7$--mani\-fold~$Y_t$.%

\begin{remark}
  The diffeomorphism type of $Y_t$ is independent of $t>0$.  Hence, we
  will sometimes drop the label $t$ and pretend to be working with a
  fixed $7$--manifold $Y$.  However, at various points it will be
  important to remember the precise way in which $Y_t$ was
  constructed.
\end{remark}

\begin{remark}\label{rmk:topy}
  The (co)homology groups and the fundamental group of $Y$
  can relatively easily be computed from the above construction, the
  latter being especially important in view of \fullref{rmk:holpi1}.
  In particular, it can be seen that every $\Sigma\in H_2(X_j,\Z)$
  invariant under the action of $H_j$ yields a cohomology class
  $\Sigma\in H_2(Y,\Z)$.  Also each component of singular set $S_j$
  gives rise to a \emph{rational} homology class
  \begin{equation}\label{eq:Sj}
    [S_j]:=\frac{1}{|H_j|} (\iota_{j,t})_*\big(T^3\times\{x\}\big) \in H_3(Y,\Q),
  \end{equation}
  where $\iota_{j,t}\co T^3\times
  \pi_{\smash{j,t}}^{-1}(B_{\smash{\zeta}}^4/G_j) \to Y$ denotes the projection
  to $\tilde T_{j,t}$ followed by the inclusion into $Y$ and $x$
  denotes a point in $\pi_{j,t}^{-1}(B_\zeta^4/G_j)$.
\end{remark}

On $\tilde T_{j,t}$ there is a torsion-free $\rG_2$--structure given by
\begin{equation*}
  \hat\phi_{j,t}:=\delta^1\wedge\delta^2\wedge\delta^3
  +t^2\delta^1\wedge\hat\omega_{j,1}
  +t^2\delta^2\wedge\hat\omega_{j,2}
  -t^2\delta^3\wedge\hat\omega_{j,3}.
\end{equation*}
Near the boundary of $\tilde T_{j,t}$ the $3$--forms $\hat\phi_{j,t}$
and $\phi_0$ are close to each other.  In order to patch them together
note that there are $1$--forms $\varrho_{j,t,i}$ on
$(\C^2\setminus\{0\})/G_j$ such that
\begin{equation*}
  t^2(\pi_{j,t})_*\hat\omega_{j,i}=\omega_i+\rd\varrho_{j,t,i}
\end{equation*}
with $\nabla^k \varrho_{j,t,i} = t^4O(r^{-3-k})$ for $k\geq 0$; see
Joyce~\cite[Theorem~8.2.3]{Joyce2000}.  Now, fix a smooth non-decreasing
function $\chi\co[0,\zeta]\to[0,1]$ such that $\chi(s)=0$ for
$s\leq\zeta/4$ and $\chi(s)=1$ for $s\geq\zeta/2$ and set
\begin{equation*}
  \tilde\omega_{j,t,i}:= t^2\hat\omega_{j,i}
 - \rd(\chi(|\pi_{j,t}|)\cdot \pi_{j,t}^*\varrho_{j,t,i}).
\end{equation*}
Then $(\pi_{j,t})_*\tilde\omega_{j,t,i}$ and $\omega_i$ agree on
$r^{-1}[\zeta/2,\infty)$ and we can define a $3$--form
$\tilde\phi_t\in\Omega^3(Y_t)$ by $\tilde\phi_t:=\phi_0$ on
$Y_0\setminus T_t=Y_t\setminus\tilde T_t$ and by
\begin{equation*}
  \tilde\phi_t:=\delta^1\wedge\delta^2\wedge\delta^3
  +\delta^1\wedge\tilde\omega_{j,t,1}
  +\delta^2\wedge\tilde\omega_{j,t,2}
  -\delta^3\wedge\tilde\omega_{j,t,3}
\end{equation*}
on $\tilde T_{j,t}$.  Define the function $r_t\co Y_t\to [0,\zeta]$ by
\begin{equation}\label{eq:rt}
  r_t(p):=
  \begin{cases}
    |\pi_{j,t}(y)| & \text{for}~ p=[(x,y)] \in \tilde T_{j,t} \\
    \zeta & \text{for}~ p \in Y_t \setminus \tilde T_t
  \end{cases}
\end{equation}
and set
\begin{equation}\label{eq:Rt}
  R_{j,t}:=\tilde T_{j,t}\cap r_t^{-1}[\zeta/4,\zeta/2]
  \quad\text{and}\quad
  R_t:=\bigcup_j R_{j,t}=r_t^{-1}[\zeta/4,\zeta/2].
\end{equation}
Outside $R_t$ the $3$--form $\smash{\tilde\phi_t}$ defines a torsion-free
$\rG_2$--structure, while on $R_{j,t}$ it satisfies
$\nabla^k(\tilde\phi_t-\hat\phi_{j,t})=O(t^4)$ for $k\geq 0$ and
similarly, for each fixed $\epsilon>0$, on $r_t^{-1}[\epsilon,\zeta]$
we have $\nabla^k(\tilde\phi_t-\phi_0)=O(t^4)$ for $k\geq 0$.  In
particular, $\tilde\phi_t$ defines a $\rG_2$--structure on $Y_t$
provided $t>0$ is sufficiently small.

We equip $Y_t$ with the Riemannian metric $\tilde g_t:=g_{\tilde
  \phi_t}$ associated with $\tilde\phi_t$.

\begin{remark}\label{rmk:mcmp}
  Note that on the complement of $\tilde T_t$ the metric $\tilde g_t$
  agrees with the flat metric $g_0$ on $(T^7/\Gamma)\setminus T$ and
  on $\tilde T_{j,t}\setminus R_{j,t}$ it agrees with the metric
$$g_{\hat\phi_{j,t}}=g_{\R^3}\oplus t^2g_{X_j}.$$  
Here $g_{\smash{\R^3}}$
  denotes the standard metric on $\R^3$ and $g_{\smash{X_j}}$ denotes the
  metric on $X_j$.  Moreover, since the map $\phi\mapsto g_\phi$ is
  smooth, on $R_{j,t}$ we have $\nabla^k(\tilde g_t-g_{\smash{\R^3}}\oplus
  t^2g_{\smash{X_j}})=O(t^4)$ for $k\geq 0$ and, for each fixed
  $\epsilon>0$, on $r_t^{-1}[\epsilon,\zeta]$ we have
  $\nabla^k(\tilde g_t - g_0)=O(t^4)$ for $k\geq 0$.
\end{remark}

\begin{theorem}[Joyce~{\cite[Theorems~A and~B]{Joyce1996a}};
{\cite[Theorem~2.2.1]{Joyce1996b}}]\label{thm:joyce}
  There are constants $T,c>0$ and for each $t\in(0,T)$ a $2$--form
  $\eta_t$ on $Y_t$ such that $\phi_t:=\tilde\phi_t+\rd\eta_t$ defines
  a torsion-free $\rG_2$--structure and
  \begin{equation}\label{eq:dee}
    \|\rd\eta_t\|_{L^\infty} \leq ct^{1/2}.
  \end{equation}
\end{theorem}

\begin{remark}
  In view of \fullref{thm:fg} the above is tantamount to saying
  that one can solve the non-linear partial differential equation
  \begin{equation}\label{eq:j}
    \rd\Theta\big(\tilde\phi_t+\rd\eta_t\big)=0
  \end{equation}
  with estimates on $\rd\eta_t$.  For small $\eta_t$, the dominant part
  of this equation is essentially the Laplacian on $2$--forms.  Now, as
  $t>0$ decreases the size of $\rd\Theta(\tilde\phi_t)$ becomes
  smaller and smaller, but at the same time the mapping properties of
  the Laplacian degenerate.  Solving \eqref{eq:j} thus is a rather
  delicate balancing act.
\end{remark}

For our application we need to slightly strengthen the estimate in
\fullref{thm:joyce}.  Let $w_t(x,y):=t+\min\{r_t(x),r_t(y)\}$.
For a H\"older exponent $\alpha\in(0,1)$ define
\begin{align*}
  [f]_{C^{0,\alpha}_{0,t}(U)}
  &:= \sup_{d(x,y)\leq w_t(x,y)} w_t(x,y)^\alpha
  \frac{|f(x)-f(y)|}{d(x,y)^\alpha}, \\[1ex]
  \|f\|_{C^{0,\alpha}_{0,t}(U)} 
  &:= \|f\|_{L^\infty(U)} + [f]_{C^{0,\alpha}_{0,t}(U)},
\end{align*}
for a tensor field $f$ over $U\subset Y_t$. Here we use parallel
transport to compare the values of $f$ at various points of $U$.
If $U$ is unspecified, then we take $U=Y_t$.

\begin{prop}\label{prop:joyce+}
  The constants $T,c>0$ in \fullref{thm:joyce} can be chosen such
  that for all $t\in(0,T)$ we have
  \begin{equation*}
    \|\rd\eta_t\|_{C^{0,\alpha}_{0,t}} \leq ct^{1/2}
    \quad\text{and}\quad
    \bigl\lVert
\Theta(\phi_t)-\Theta\big(\hat\phi_{j,t}\big)
\bigr\rVert_{C^{0,\alpha}_{0,t}(\tilde T_{j,t})} \leq ct^{1/2}.
  \end{equation*}
\end{prop}

For the proof of this result it will be helpful to note the following.

\begin{prop}\label{prop:mcmp}
  For each $\mu>0$ and $K\in\N_0$ there exists a constant $\epsilon>0$
  such that the following holds for all $t\in(0,T)$ and $p\in Y_t$:
  $R:=\epsilon(t+r_t(p))$ is less than the injectivity radius of
  $(Y_t,\tilde g_t)$ at $p$ and if we identify $T_pY$ isometrically
  with $\R^7$ and denote by $s_R\co B_1\to B_R(p)$ the map obtained by
  multiplication with $R$ followed by the exponential map, then
  \begin{equation}\label{eq:mcmp}
    \big|\del^k\big(R^{-2}s_R^*\tilde g_t - g_{\R^7}\big)\big|\leq \mu
  \end{equation}
  for all $k\in\{0,\ldots,K\}$.  Here $g_{\R^7}$ denotes the standard
  metric on $\R^7$.
\end{prop}

\begin{proof}
  From \fullref{rmk:mcmp} it is clear that we can find $\epsilon>0$
  such that the above statement holds for all $p \in
  r_t^{-1}[\zeta/8,\zeta]$.  Moreover, for $p\in
  r_t^{-1}[0,\zeta/8]$ inequality \eqref{eq:mcmp} is equivalent to
  \begin{equation*}
    \big|\del^k\big({\tilde R}^{-2}s_{\tilde R}^*(g_{\R^3}\oplus g_{X_j}) -
      g_{\R^7}\big)\big|\leq \mu,
  \end{equation*}
  where $\tilde R:=\epsilon(1+|\pi_{j}(y)|)$ and $p=[(x,y)]$. Because
  of \eqref{eq:aledecay} this holds for all $\epsilon\leq \frac12$ as
  long as $|\pi_{j}(y)|$ is sufficiently large, say, $|\pi_{j}(y)|>N$.
  For $|\pi_{j}(y)| \leq N$ it can be arranged to hold by choosing
  $\epsilon>0$ sufficiently small.
\end{proof}

\begin{proof}[Proof of \fullref{prop:joyce+}]
  Note that the second part follows from the first and the
  construction of $\tilde\phi_t$, because $\Theta$ is a smooth map.
  To obtain the estimate on $\rd\eta_t$ recall from Joyce's
  construction that $\eta_t$ solves a non-linear partial differential
  equation that can be written schematically as
  \begin{equation}\label{eq:ddP}
    \rd^*\rd \eta_t + P(\rd\eta_t,\nabla\rd\eta_t) =
    G(\rd\eta_t,\ldots@) \quad\text{and}\quad \rd^*\eta_t=0;
  \end{equation}
  see Joyce~\cite[Equation~(33)]{Joyce1996a}.  The crucial points are
  that $P(x,y)$ is a smooth function which depends linearly on $y$ and
  satisfies $P(0,y)=0$ and that there is a constant $c>0$ such that
  \begin{equation}\label{eq:G}
    \|G(\rd\eta_t,\ldots@)\|_{L^\infty}\leq ct^{1/2}.
  \end{equation}
  Now, define
  \begin{equation*}
    D_t\sigma:=(\rd^*\sigma+P(\rd\eta_t,\nabla\sigma),\rd\sigma).
  \end{equation*}  
  Since $\rd\eta_t$ is small provided $T>0$ is small, this a small
  perturbation of the operator $\rd^*\oplus\rd$.  We extend $D_t$ to
  an operator from $\Omega^*(Y_t)$ to itself by defining
  $D_t\sigma=(\rd^*\oplus\rd)\sigma$ for $\sigma\in\Omega^k(Y_t)$ with
  $k\neq 3$, so that it becomes an elliptic operator.  We will now
  prove that there are constants $c>0$ and $\epsilon\in(0,\frac12)$
  such that for all $t\in(0,T)$ and each $p\in Y_t$ the following
  holds:
  \begin{equation}\label{eq:dts}
    R^\alpha [\sigma]_{C^{0,\alpha}(B_{R/2}(p))}
    \leq c \big(R\|D_t \sigma\|_{L^\infty(B_{R}(p))} +
\|\sigma\|_{L^\infty(B_{R}(p))}\big)
  \end{equation}
  with $R:=\epsilon(t+r_t(p))$.  From this the asserted bound on
  $[\rd\eta_t]_{@\smash{C^{0,\alpha}_{0,t}}}$ follows at once using
  \eqref{eq:dee}, \eqref{eq:ddP} and \eqref{eq:G}, since on
  $B_{R/2}(p)$ we have $w_t\leq 2\epsilon^{-1}R$.

  For $\mu>0$ choose $\epsilon>0$ according to \fullref{prop:mcmp}
with $K=1$.  Let $s_R\co \smash{B_1^7}\to B_R(p)$ be as in
\fullref{prop:mcmp}.  We define a rescaled operator $\tilde
  D_{t,p}\co\Omega^*(B_1)\to\Omega^*(B_1)$ by
  \begin{equation*}
    \tilde D_{t,p} \sigma := \big(R^2 s_R^*\tau,s_R^*\theta\big)
  \end{equation*}
  for $\sigma\in\Omega^k(B_1)$, where
  $(\tau,\theta):=D_t(s_R^{-1})^*\sigma \in
  \Omega^{k-1}(B_1)\oplus\Omega^{k+1}(B_1)$.  It follows from
  \fullref{thm:joyce} and \fullref{prop:mcmp} that by
  choosing $T,\mu>0$ sufficiently small, we can arrange that for all
  $t\in(0,T)$ and $p\in Y_t$ the rescaled operator $\tilde D_{t,p}$ is
  as close to $\rd\oplus \rd^*\co \Omega^*(B_1)\to\Omega^*(B_1)$ as we
  wish.  In particular, we can arrange that the family of operators
  $\tilde D_{t,p}$ is uniformly elliptic with coefficients uniformly
  bounded in $C^1$.  Hence, by standard elliptic theory, we can find a
  constant $c>0$ independent of $t\in(0,T)$ and $p\in Y_t$ such that
  the following $L^q$ estimate holds:
  \begin{equation*}
    \|\sigma\|_{W^{1,q}(B_{1/2})}\leq c \big(\|\tilde D_{t,p}
    \sigma\|_{L^q(B_1)} + \|\sigma\|_{L^q(B_1)}\big).
  \end{equation*}
  Combined with the Sobolev embedding $W^{1,q}\into C^{0,1-7/q}$ this
  yields
  \begin{equation*}
    [\sigma]_{C^{0,\alpha}(B_{1/2})}\leq c \big(\|\tilde
    D_{t,p}\sigma\|_{L^\infty(B_1)} + \|\sigma\|_{L^\infty(B_1)}\big)
  \end{equation*}
  with $c>0$ independent of $t\in(0,T)$ and $p\in Y_t$.  This,
  however, is equivalent to the estimate \eqref{eq:dts} for the
  unscaled operator $D_t$.
\end{proof}

\begin{remark}
  \fullref{prop:joyce+} can be viewed as a quantification of
  Joyce's proof of the fact that $\eta_t$ is smooth.  In a similar
  fashion, one can also obtain estimates on higher H\"older norms of
  $\rd\eta_t$.
\end{remark}

\begin{remark}\label{rmk:nw}
  The kind of argument we used above goes back to work of
  Nirenberg--Walker~\cite[Theorem~3.1]{Nirenberg1973}.  We will
  encounter this line of reasoning again in the proofs of Propositions
  \ref{prop:decay} and \ref{prop:mse}.
\end{remark}

\section{ASD instantons on ALE spaces}
\label{sec:asdale}

Let $\Gamma$ be a finite subgroup of $\SU(2)$, let $X$ be an ALE space
asymptotic to $\C^2/\Gamma$ and let $E$ be a $G$--bundle over $X$.  We
denote by $\sA(E)$ the space of connections on $E$.

\begin{definition}
  A \emph{framing at infinity} of $E$ is a bundle isomorphism $\Phi\co
  E_\infty|_U \to \pi_*E|_U$ where $E_\infty$ is a $G$--bundle over
  $(\C^2\setminus \{0\})/\Gamma$ and $U$ is the complement of a
  compact neighbourhood of the singular point in $\C^2/\Gamma$.
\end{definition}

Let $\theta$ be a flat connection on a $G$--bundle $E_\infty$ over
$(\C^2\setminus\{0\})/\Gamma$.

\begin{definition}
  Let $\Phi\co E_\infty|_U \to \pi_*E|_U$ be a framing at infinity of
  $E$.  Then a connection $A\in\sA(E)$ is called \emph{asymptotic to
    $\theta$ at rate $\delta$ with respect to $\Phi$} if
  \begin{equation}\label{eq:Adecay}
    \nabla^k (\Phi^*A-\theta)=O\big(r^{\delta-k}\big)
  \end{equation}
  for all $k\geq 0$.  Here $\nabla$ is the covariant derivative
  associated with $\theta$.
\end{definition}

\begin{definition}
  A \emph{framed ASD instanton asymptotic to $\theta$ (at rate
    $\delta$)} is an ASD instanton $A\in\sA(E)$ on $E$ together with a
  framing at infinity $\Phi$ of $E$ such that $A$ is asymptotic to
  $\theta$ at rate $\delta$ with respect to $\Phi$.  If no rate
  $\delta$ is specified, then we take $\delta=-3$.
\end{definition}

\begin{prop}\label{prop:Adecay}
  Let $A\in\sA(E)$ be an ASD instanton on $E$ with finite energy, that
  is,
  \begin{equation*}
    \int_X |F_A|^2 \dvol < \infty,
  \end{equation*}
  then there is a $G$--bundle $E_\infty$ over
  $(\C^2\setminus\{0\})/\Gamma$ together with a flat connection
  $\theta$ and a framing $\Phi\co E_\infty|_U \to \pi_*E|_U$ such that
  \eqref{eq:Adecay} holds with $\delta=-3$
\end{prop}

\begin{proof}
  We extend the argument in Donaldson--Kronheimer~\cite[page~98]{Donaldson1990}.  The
  topological space $\hat X:=X\cup\{\infty\}$ can be given the
  structure of an orbifold whose atlas contains the charts of $X$ as
  well as a uniformising chart at infinity $\varphi\co
  B_\epsilon/\Gamma \to \hat X$ which is constructed as follows.  Fix
  an orientation reversing linear isometry $\sigma$ of $\R^4$.  We let
  $\Gamma$ act on $B_\epsilon$ by $(g,x)\mapsto
  \sigma^{-1}(g\cdot\sigma(x))$ and define $\varphi(0):=\infty$ and
  $\varphi(x)=\pi^{-1}(\sigma(x)/|x|^2)$.  If $g$
  denotes the metric on $X$, then the conformally equivalent metric
  $\hat g:=(1+|\pi|^2)^{-2}g$ extends to $\hat X$ as an orbifold
  metric.  The metric is not necessarily smooth, but only
  $C^{3,\alpha}$; however, that does not cause any problems.  One
  should think of $\hat X$ as a conformal compactification of $X$ in
  the same way that $S^4$ is a conformal compactification of $\R^4$.
  
  Since the equation $\smash{F_A^+}=0$ as well as the energy are conformally
  invariant, we can think of $A$ as a finite energy ASD instanton on
  $(\hat X\setminus\{\infty\},\hat g)$.  By Uhlenbeck's removable
  singularities theorem~\cite[Theorem~4.1]{Uhlenbeck1982}, the
  pullback of $A$ to $B_\epsilon\setminus\{0\}$ extends to a
  $\Gamma$--invariant ASD instanton over all of $B_\epsilon$.  Hence,
  $A$ extends to an ASD instanton $\hat A$ on an orbifold $G$--bundle
  $\hat E$ over $\hat X$.  Using radial parallel transport from
  $\infty$ we obtain a trivialisation of $\hat E$ over
  $\varphi(B_\epsilon/\Gamma)$ in which the connection matrix
  representing $\hat A$ vanishes at $\infty=\varphi(0)$.  Denote by
  $\rho\co\Gamma\to G$ the monodromy representation associated with
  $\hat E|_\infty$.  Associated with $\rho$ there are a $G$--bundle
  $E_\infty$ over $\varphi((B_\epsilon\setminus\{0\})/\Gamma)$ and a flat
  connection $\theta$ on $E_\infty$.  The above trivialisation of
  $\hat E$ over $\varphi(B_\epsilon/\Gamma)$ amounts to a bundle
  isomorphism $\Phi\co E_\infty \to \hat
  E|_{\smash{\varphi(B_\epsilon\setminus\{0\}/\Gamma)}}$ and the fact that the
  connection matrix representing $\hat A$ vanishes at
  $\infty=\varphi(0)$ implies that
$\nabla^k(\varphi^*(\Phi^*(``\hat
  A\:)-\theta))=O(x^{1-k})$ for all $k\geq 0$.  By considering the
  action of the inversion $x\mapsto\sigma(x)/|x|^2$ on $k$--fold
  derivatives of $1$--forms one sees that $\nabla^k
  (\Phi^*A-\theta)=O(r^{-3-k})$.
\end{proof}

Let us briefly discuss moduli spaces of framed ASD instantons on $E$
asymptotic to $\theta$.  For a detailed discussion we refer the reader
to Nakajima's beautiful article~\cite{Nakajima1990}.  Fix a framing at
infinity $\Phi$ of $E$, a rate $\delta\in(-3,-1)$ and denote by
$\sA(E,\theta)$ the space of all connections asymptotic to $\theta$ at
rate $\delta$ with respect to $\Phi$.  Similarly, define $\sG(E)$ to
be the group of gauge transformations asymptotic to a constant element
of $G$ at infinity at rate $\delta+1$ with respect to $\Phi$.  Denote
by $g_\infty\co \sG(E)\to G$ the homomorphism assigning to each gauge
transformation its asymptotic value at infinity and let
$\sG_0(E):=\ker g_\infty \subset\sG(E)$ be the based gauge group
consisting of gauge transformations asymptotic to the identity. Then
the space
\begin{equation*}
  M(E,\theta):=\{A \in \sA(E,\theta) : F_A^+=0\}/\sG_0(E)
\end{equation*}
is called the \emph{moduli space of framed ASD instantons on $E$
  asymptotic to $\theta$}.

\begin{remark}
  The space does not depend on the choice of $\delta\in(-3,-1)$.  This
  is a consequence of \fullref{prop:Adecay}.
\end{remark}

\begin{remark}
  If we denote by $\rho\co\Gamma\to G$ the monodromy representation
  associated with $\theta$ and by $G_\rho:=\left\{g\in G: g\rho
    g^{-1}=\rho\right\}$ the stabiliser of $\rho$, then $G_\rho\subset
  G\iso\sG(E)/\sG_0(E)$ acts on $M(E,\theta)$.
\end{remark}

\begin{theorem}[Nakajima {\cite[Theorem~2.6 and
Proposition~5.1]{Nakajima1990}}]
  The moduli space $M(E,\theta)$ is a smooth hyperk\"ahler manifold.
\end{theorem}

Formally, this can be seen as an infinite-dimensional instance of a
hyperk\"ahler reduction (see
Hitchin--Karlhede--Lindstr\"om--Ro\v{c}ek~\cite{Hitchin1987}).  The space $\sA(E,\theta)$
inherits a hyperk\"ahler structure from $X$ and the action of the based
gauge group $\sG_0$ has a hyperk\"ahler moment map given by
$\mu(A)=\smash{F_A^+}$.  To make this rigorous one needs to set up a suitable
Kuranishi model for $M(E,\theta)$ along the lines of
Donaldson--Kronheimer~\cite[Section~4.2.5]{Donaldson1990}.  This can be done using weighted
Sobolev space completions of $\sA(E,\theta)$ and $\sG_0(E)$; see
Nakajima~\cite[Section 2]{Nakajima1990} for a detailed discussion.  An
important role is played by the operator $\delta_A\co
\Omega^1(X,\frg_E)\to\Omega^0(X,\frg_E)\oplus\Omega^+(X,\frg_E)$
defined by
\begin{equation}\label{eq:deltaA}
  \delta_A(a):=\big(\rd_A^*a, \rd_A^+a\big)
\end{equation}
which governs the infinitesimal deformation theory of the ASD
instanton $A$.

\begin{prop}\label{prop:decay}
  Let $A\in\sA(E)$ be a finite energy ASD instanton on $E$.  Then the
  following holds.
  \begin{enumerate}
  \item\label{d:1} If $a\in\ker\delta_A$ decays to zero at infinity,
    then $\nabla_A^k a=O(|\pi|^{-3-k})$ for all $k\geq 0$.
  \item\label{d:2} If $(\xi,\omega)\in\ker\delta_A^*$ decays to zero
    at infinity, then $(\xi,\omega)=0$.
  \end{enumerate}
\end{prop}

\begin{remark}
  From the second part of this proposition one can deduce that the
  deformation theory of framed finite energy ASD instantons is always
  unobstructed; hence, $M(E,\theta)$ is a smooth manifold (see also
  \cite[Proposition~5.1]{Nakajima1990}).  By the first part the
  tangent space of $M(E,\theta)$ at $[A]$ agrees with the $L^2$ kernel
  of $\delta_A$ and thus the formal hyperk\"ahler structure is indeed
  well-defined.
\end{remark}

The proof of \fullref{prop:decay} rests on the following
refined Kato inequality.

\begin{prop}\label{prop:kato}
  Let $A\in\sA(E)$ be an ASD instanton on $E$.  If
  $a\in\Omega^1(X,\frg_E)$ satisfies $\delta_Aa=0$, then
  \begin{equation}\label{eq:kato}
    |@\rd|a|@| \leq \sqrt{\!\tfrac34}@ \big|\nabla_A a\big|
  \end{equation}
  on the complement of the vanishing locus of $a$.
\end{prop}

\begin{proof}
  Recall that the Kato inequality follows from the Cauchy--Schwarz
  inequality $|\<\nabla_A a,a\>|\leq |\nabla_A\alpha||\alpha|$.  If
  $\delta_Aa=0$, then it is not hard to see that equality can only
  hold if $\nabla_A a=0$.  This shows that \eqref{eq:kato} holds with
  some constant $\epsilon<1$ instead of $\sqrt{3/4}$.

  To see that one can take $\epsilon=\sqrt{3/4}$ we follow an argument
  of Feehan~\cite[Section~3]{Feehan2001}; however, also note that we
  could simply read off the value from the table given in
  Calderbank~\cite[Appendix]{Calderbank2000}.  We can write $\delta_A$ as a
  Dirac-type operator
  \begin{equation*}
     \delta_A a = \sum_{i} \gamma(e_i)\nabla^A_{e_i}a.
  \end{equation*}
  Here $(e_i)$ is a local orthonormal frame and the Clifford
  multiplication $\gamma$ is defined by $\gamma(v)a := (-i_va,
  (v^*\wedge a)^+)$, where $v^*$ denotes the dual of $v$ with respect
  to the metric on $X$.  For $x\in X$ with $a(x) \neq 0$ and $\rd
  |a|(x) \neq 0$ pick an orthonormal basis $(e_i)$ of $T_xX$ with
  $e_1:= \nabla |a|/|\nabla |a||$.  Since $\delta_Aa=0$ and
  $|\gamma(v)a|=|v||a|$, we have
  \begin{equation*}
    \big|\rd |a|\big|^2=\big|\nabla_{e_1} |a|\big|^2\leq
  \big|\nabla^A_{e_1} a\big|^2
    = \big|\gamma(e_1)\nabla^A_{e_1}a\big|^2
    = \Big|\sum\nolimits_{i\geq2} \gamma(e_i)\nabla^A_{e_i}a\Big|^2
    \leq 3 \sum_{i\geq 2} \big|\nabla^A_{e_i}a\big|^2
  \end{equation*}
  and therefore
  \begin{equation*}
    4\big|\rd|a|\big|^2=4 \big|\nabla^A_{e_1} a\big|^2 \leq
    3 \sum_{i} \big|\nabla^A_{e_i}a\big|^2 =
    3\big|\nabla_Aa\big|^2.
  \end{equation*}
  This finishes the proof.
\end{proof}

\begin{proof}[Proof of \fullref{prop:decay}]
  First of all note that \eqref{d:1} implies \eqref{d:2}, because if
  $\delta_A^*(\xi,\omega)=0$, then 
$\smash{\rd_A^*}\rd_A\xi=0$ and
$\smash{\rd_A^+}\rd_A\xi=[\smash{F_A^+},\xi]=0$; therefore $\rd_A
  \xi=O(|\pi|^{-3})$. Thus integration by parts yields $\rd_A\xi=0$
  and, hence, $\xi=0$.  Similarly, one shows that $\omega=0$.

  We will first explain why \eqref{d:1} for $k=0$ implies the asserted
  estimates for $k>0$ as well.  The argument is similar to that in
  \fullref{prop:joyce+}.  For $x\in X$ set
  $R:=\frac12(1+|\pi(x)|)$.  We claim that there is a constant
  $c=c(k)>0$ independent of $x\in X$ such that
  \begin{equation}\label{eq:dse}
    R^{k}\|\nabla^k_A a\|_{L^\infty(B_{R/2}(x))} \leq c \|a\|_{L^\infty(B_R(x))}
  \end{equation}
  for all $a\in\ker\delta_A$.  This clearly implies \eqref{d:1} for $k>0$
  given the statement for $k=0$.  For $|\pi(x)|$ sufficiently large,
  say $|\pi(x)|>R_0$, the restriction of $A$ to $B_R(x)$ is
  arbitrarily close to a flat connection by \fullref{prop:Adecay}.  We rescale to a ball of radius one and denote
  the rescaled connection by $\tilde A$ and the rescaling of
  $\delta_A$ by $\tilde D_x$.  Then the family of operators $\tilde
  D_x$ is uniformly elliptic with coefficients uniformly bounded in $C^1$.
  Therefore, there is a constant $c>0$ independent of $x\in X$ such
  that the following Schauder estimates holds:
  \begin{equation*}
    \big\|\nabla^k_{\tilde A}\, a\big\|_{L^\infty(B_{1/2})} \leq
    c\big(\|\tilde D_x a\|_{C^{k,\alpha}(B_1)}+\|a\|_{L^\infty(B_1)}\big).
  \end{equation*}
  If $a$ is in the kernel of $\tilde D_x$, the first term vanishes.
  Rescaling this inequality yields \eqref{eq:dse} for
  $a\in\ker\delta_A$ and $|\pi(x)|>R$.  For $1/2\leq |\pi(x)|\leq
  R_0$, \eqref{eq:dse} follows from standard Schauder estimates.

  Let us now prove \eqref{d:1} for $k=0$.  Recall, for example, from
  Freed--Uhlenbeck~\cite[Equation~(6.25)]{Freed1991}, that the operator
  $\tilde\delta_A\co
  \Omega^1(X,\frg_E)\to\Omega^0(X,\frg_E)\oplus\Omega^+(X,\frg_E)$
  defined by $\tilde\delta_A(a):=(\rd_A^*a, \sqrt2@\rd_A^+a)$
  satisfies a Weitzenb\"ock formula of the form
  \begin{equation}\label{eq:wb}
    \tilde\delta_A^*\tilde\delta_A a
    = \nabla_A^*\nabla_A a + \{\Ric,a\} + \{F_A^-,a\}.
  \end{equation}
  Here $\{\,\cdot\,,\,\cdot\,\}$ denote certain universal bilinear forms,
  whose precise form, however, is not important for our purposes and
  $\Ric$ denotes the Ricci tensor of $X$.  In our situation, since $X$
  is hyperk\"ahler and thus Ricci flat, the second term vanishes.  Now,
  suppose that $\delta_Aa=0$ and thus $\tilde\delta_Aa=0$.  Then
  \fullref{prop:kato}, the identity
$$
\Delta|a|^2+2|\nabla_Aa|^2=2\<a,\nabla_A^*\nabla_Aa\>
$$ 
(see~\cite[Equation~(6.18)]{Freed1991}) and the Weitzenb\"ock formula
  \eqref{eq:wb} yield the following estimate on the complement of the
  vanishing locus of $a$:
  \begin{align*}
    3 \Delta |a|^{2/3}
   &\leq|a|^{-4/3}\big(\Delta|a|^2+\tfrac83\big|@\rd|a|@\big|^2\big) \\
   &\leq |a|^{-4/3}\big(\Delta|a|^2+2|\nabla_A a|^2\big) \\
   &=2|a|^{-4/3}\big\langle a,\nabla_A^*\nabla_Aa\big\rangle \\
   &=2|a|^{-4/3}\big(\big\langle\tilde\delta_A^*\tilde\delta_A
a,a\big\rangle+\big\langle\big\{F_A^-,a\big\},a\big\rangle\big) \\
   &\leq O(|\pi|^{-4}) |a|^{2/3}.
  \end{align*}
  In the last step we used $\tilde\delta_A a=0$ and
  $|F_A^-|=O(|\pi|^{-4})$, which is a consequence of
  \fullref{prop:Adecay}.

  Now, let $U:=\{x\in X : a(x)\neq 0\}$ and set $f:=|a|^{2/3}$.  We
  will show that $f=O(|\pi|^{-2})$ which is equivalent to the
  desired decay estimate for $a$.  It follows from the above that on
  $U$,
  \begin{equation*}
    \Delta f \leq \frac{cf}{1+|\pi|^4}
  \end{equation*}
  for some constant $c>0$.  Since $f$ is bounded, by
  Joyce~\cite[Theorem~8.3.6(a)]{Joyce2000}, there is a $g=O(|\pi|^{-1})$
  such that
 \begin{equation*}
    \Delta g=
    \begin{cases}
      (\Delta f)^+ & \text{on}~U\\ 
      0 & \text{on}~X\setminus U.
    \end{cases}
  \end{equation*}
  Here $(\,\cdot\,)^+$ denotes taking the positive part.  Since $g$ is
  superharmonic and decays to zero at infinity, the maximum principle
  implies that $g$ is non-negative.  The function $f-g$ is a
  subharmonic on $U$, decays to zero at infinity and is non-positive
  on the boundary of $U$; hence, by the maximum principle $f\leq g$
  and thus $f\leq g=O(|\pi|^{-1})$.  Now, $(\Delta
  f)^+=O(|\pi|^{-5})$ on $U$ and an application of~\cite[Theorem~8.3.6(b)]{Joyce2000} shows that we could, in fact,
  have chosen $g$ such that $g=O(|\pi|^{-2})$.  It follows that
  $f=O(|\pi|^{-2})$ as desired.
\end{proof}

The dimension of $M(E,\theta)$ can be computed using the following
index formula.
\vadjust{\goodbreak}

\begin{theorem}[Nakajima~{\cite[Theorem~2.7]{Nakajima1990}}]\label{thm:index}
  Let $A$ be a framed ASD instanton asymptotic to $\theta$.  Then the
  dimension of the $L^2$ kernel of $\delta_A$ is given by
  \begin{equation}\label{eq:index}
    \dim\ker\delta_A=-2\int_X p_1(\frg_E)
    +\frac{2}{|\Gamma|}\sum_{g\in\Gamma\setminus\{e\}}
    \frac{\chi_\frg(g)-\dim\frg}{2-\tr g}.
  \end{equation}
  Here $p_1(\frg_E)$ is the Chern--Weil representative of the first
  Pontryagin class of $E$ and $\chi_\frg$ is the character of $\Gamma$
  acting on $\frg$, the Lie algebra associated with $G$, via the
  monodromy representation $\rho\co\Gamma \to G$ of $\theta$.
\end{theorem}

\begin{proof}
  Let us briefly explain how to derive \eqref{eq:index} from
  Nakajima's formula, which can be written as
  \begin{multline}\label{eq:nif}
    \dim\ker\delta_A
   =-\int_X \(\dim\frg+p_1(\frg_E)\)\,\ch\(S^+\)\hat A(X) \\[-1ex]
    + \dim \frg^\Gamma
    + \frac{1}{|\Gamma|}\sum_{g\in\Gamma\setminus\{e\}}
       \chi_{\frg}(g)\frac{\tr g}{2-\tr g}.
  \end{multline}
  Here $\frg^\Gamma$ denotes the $\Gamma$--invariant part of $\frg$,
  $S^+$ denotes the positive spin bundle on $X$, and $\ch(S^+)$ and
  $\hat A(X)$ denote the Chern--Weil representatives of the Chern
  character of $S^+$ and the $\hat A$--genus of $X$, respectively.

  If $A$ is the product connection on the trivial bundle rank $1$
  bundle and $a$ lies in the $L^2$ kernel of $\delta_A$, then it
  follows from the fact that $X$ is Ricci-flat and the Weitzenb\"ock
  formula \eqref{eq:wb} that $\nabla^*\nabla a=0$ and then by
  integration by parts, which is justified because of the decay
  asserted by \fullref{prop:decay}, that $\nabla a=0$.  Since
  $a$ lies in $L^2$, it necessarily vanishes.  Therefore
  $\dim\ker\delta_A=0$ and \eqref{eq:nif} yields
  \begin{equation*}
    \int_X \ch(S^+)\hat A(X)=1+\frac{1}{|\Gamma|}\sum_{g\in\Gamma\setminus\{e\}}
       \frac{\tr g}{2-\tr g}.
  \end{equation*}
  By plugging this back into \eqref{eq:nif} we obtain
  \begin{equation*}
    \dim\ker\delta_A
    =-2\int_X p_1(\frg_E)
    + \dim \frg^\Gamma - \dim \frg
    + \frac{1}{|\Gamma|}\sum_{g\in\Gamma\setminus\{e\}}
       (\chi_{\frg}(g)-\dim \frg)\frac{\tr g}{2-\tr g}.
  \end{equation*}
  Since
  \begin{equation*}
    \frac{1}{|\Gamma|}\sum_{g\in\Gamma}
       (\chi_{\frg}(g)-\dim \frg)=\dim\frg^\Gamma - \dim \frg,
  \end{equation*}
  this leads to the index formula \eqref{eq:index} given above.
\end{proof}

There is a very rich existence theory for ASD instantons on ALE
spaces.  Gocho--Nakajima~\cite{Gocho1992} observed that for each
representation $\rho\co \Gamma \to \U(n)$ there is a bundle $\cR_\rho$
over $X$ together with an ASD instanton $A_\rho$ asymptotic to the
flat connection determined by $\rho$, and if $\sigma$ is a further
representation of $\Gamma$, then $A_{\rho\oplus\sigma} = A_\rho \oplus
A_\sigma$.  Kronheimer--Nakajima~\cite{Kronheimer1990} took this as
the starting point for an ADHM construction of ASD instantons on ALE
spaces.  One important consequence of their work is the following
rigidity result.

\begin{definition}\label{def:rigid}
  An ASD instanton $A$ is called \emph{infinitesimally rigid} if the
  $L^2$ kernel of the linear operator $\delta_A$ is trivial.
\end{definition}

\begin{theorem}[Kronheimer--Nakajima~{\cite[Lemma~7.1]{Kronheimer1990}}]
\label{thm:Rreg}
  For each $\rho\co\Gamma\to\U(n)$ the ASD instanton $A_\rho$ is
  infinitesimally rigid.
\end{theorem}

By combining this result applied to the regular representation with
the index formula Kronheimer--Nakajima derive a geometric version of
the McKay correspondence \cite[Appendix~A]{Kronheimer1990}.  Let
$\Delta(\Gamma)$ denote the Dynkin diagram associated with $\Gamma$ in
the ADE classification of the finite subgroups of $\SU(2)$.  Each
vertex of $\Delta(\Gamma)$ corresponds to a non-trivial irreducible
representation.  We label these by $\rho_1,\ldots,\rho_k$ and denote
the associated bundles by $\cR_j$ and the associated ASD instantons by
$A_j$.

\begin{theorem}[Kronheimer--Nakajima~{\cite[Appendix~A]{Kronheimer1990}}]
\label{thm:gmk}
  The harmonic $2$--forms $c_1(\cR_j)=\frac{i}{2\pi}\tr{F_{A_j}}$ form
  a basis of $L^2\cH^2(X)\iso{}H^2(X,\R)$ and satisfy
  \begin{equation*}
    \int_X c_1(\cR_i)\wedge c_1(\cR_j)= -(C^{-1})_{ij},
  \end{equation*}
  where $C$ is the Cartan matrix associated with $\Delta(\Gamma)$.
  Moreover, there is an isometry $\kappa\in\Aut(H_2(X,\Z),\cdot\,)$ such
  that $\{c_1(\cR_j)\}$ is dual to $\{\kappa[\Sigma_j]\}$, where
  $\Sigma_j$ are the irreducible components of the exceptional divisor
\vrule width 0pt height 10.5pt depth 3.5pt
%FMT
  $E$ of $\smash{\tC\Gamma}$.  If $X$ is isomorphic to
  $\smash{\tC\Gamma}$ as a complex manifold, then $\kappa=\id$.
\end{theorem}

This result is very useful for computing the index of $\delta_A$ when
$A$ is constructed out of ASD instantons of the form $A_\rho$ (by
taking tensor products, direct sums, etc.).

\begin{prop}\label{prop:rigid}
  Let $X$ be an ALE space asymptotic to $\C^2/\Z_k$.  Denote by
  $\rho_j\co \Z_k`\to`U(1)$ the irreducible representation defined by
  $\rho_j(\ell)=\exp(\smash{\frac{2\pi i}{k}} j\ell)$.  For $n,m\in\Z_k$,
  let $E_{n,m}$ be the $\SO(3)$--bundle underlying
  $\R\oplus(\cR_n^*\otimes \cR_{n+m})$ and denote by $A_{n,m}$ the ASD
  instanton on $E_{n,m}$ induced by $A_n$ and $A_{n+m}$.  Then
  $A_{n,m}$ is infinitesimally rigid, asymptotic at infinity to the
  flat connection associated with $\rho_{m}$ and
  \begin{equation*}
    \frac{1}{8\pi^2}\int_X |F_{A_{n,m}}|^2=\frac{(k-m)m}{k}
  \end{equation*}
  as well as
  \begin{equation*}
    w_2(\frg_{E_{n,m}})=c_1(\cR_{n+m})-c_1(\cR_{n}) \in H^2(X,\Z_2). 
    \vadjust{\goodbreak}
  \end{equation*}
\end{prop}

\begin{proof}
  To see that $A_{n,m}$ is infinitesimally rigid apply
  \fullref{thm:Rreg} to $A_{n} \oplus A_{n+m}$ and observe that
  $\frg_{E_{n,m}}=\R\oplus(\cR_n^*\otimes \cR_{n+m})$ is a parallel
  subbundle of $\frg_{\cR_n\oplus\cR_{n+m}}$.

  The energy of $A_{n,m}$ can be computed using \fullref{thm:gmk} or
  by noting that the first term in the index formula \eqref{eq:index}
  is precisely twice the energy and the second term is given by
  $(-\smash{\frac2k})$--times
  \begin{equation*}
    -\sum_{g\neq e} \frac{\chi_\frg(g)-\dim\frg}{2-\tr
      g}=\sum_{j=1}^{k-1} \frac{1- \cos(2\pi mj/k)}{1-\cos(2\pi
      j/k)}=(k-m)m.
  \end{equation*}
  The statement about the second Stiefel--Whitney class is clear.
\end{proof}

\section[Approximate G_2--instantons]
{Approximate $\rG_2$--instantons}
\label{sec:approx}

Throughout this section, let $Y_0$ be an admissible $\rG_2$--orbifold,
let $\br=\{(X_j,\rho_j)\}$ be resolution data for $Y_0$ and denote by
$(Y_t,\phi_t)_{t\in(0,T)}$ the family of $\rG_2$--manifolds obtained
from $\br$ via \fullref{thm:joyce}.  Denote by
$\psi_t:=\Theta(\phi_t)$ the coassociative calibration on $Y_t$.  If
$\theta$ is a flat connection on a $G$--bundle $E_0$ over $Y_0$, then
the monodromy of $\theta$ around $S_j$ induces a representation
$\mu_j\co \pi_1(T_j,x_j)\iso(\Z^3 \times G_j) \rtimes H_j \to G$ of
the orbifold fundamental group of $T_j$ based at $x_j\in T_j\setminus
S_j$.

\begin{remark}
  For a general definition of orbifold fundamental group we refer the
  reader to Adem--Leida--Ruan~\cite[Definition~1.50 and Section~2.2]{Adem2007}.  All
  orbifold fundamental groups $\pi_1(X)$ encountered in this article
  can be identified with the fundamental groups $\pi_1(X^{\rm reg})$
  of the regular part of the orbifold in question, since the singular
  sets have sufficiently large codimension.
\end{remark}

\begin{definition}
  A collection
  $\bg=((E_0,\theta),\{(x_j,f_j)\},\{(E_j,A_j,\tilde\rho_j,m_j)\})$
  consisting of $E_0$ and $\theta$ as above as well as, for each $j$,
  the choice of
\begin{itemize}
\item a point $x_j \in T_j\setminus S_j$ together with a framing
  $f_j\co (E_0)_{x_j} \to G$ of $E_0$ at $x_j$,
\item a $G$--bundle $E_j$ over $X_j$ together with a framed ASD
  instanton $A_j$ asymptotic at infinity to the flat connection on the
  bundle $E_{\infty,j}$ over $(\C^2\setminus\{0\})/G_j$ induced by
  the representation $\mu_j|_{G_j}$,
\item a lift $\tilde\rho_j$ of the action $\rho_j$ of $H_j$ on $X_j$
  to $E_j$ and
\item  a homomorphism $m_j\co \Z^3 \to \sG(E_j)$
\end{itemize}
is called \emph{gluing data compatible with $\br=\{(X_j,\rho_j)\}$} if
the following compatibility conditions are satisfied:
\begin{itemize}
\item The action $\tilde\rho_j$ of $H_j$ on $E_j$ preserves $A_j$ and
  is asymptotic at infinity, with respect to the framing associated
  with $A_j$, to the action of $H_j$ on $E_{\infty,j}$.  Note that
  the lift of the action of $H_j$ on $E_{\infty,j}$ to the trivial
  bundle $G\times (\C^2\setminus\{0\})$ is given by
  $h\cdot(g,x)=(\mu_j(h)\cdot g,h\cdot x)$.
\item The action of $\Z^3$ on $E_j$ given by $m_j$ preserves $A_j$ and
  $m_j$ is asymptotic at infinity to $\mu_j|_{\Z^3}$, that is, $g_\infty
  \circ m_j=\mu_j|_{\Z^3}$ with $g_\infty\co \sG(E_j)\to G$ as in the
  paragraph following the proof of \fullref{prop:Adecay}.
\item For all $h\in{}H_j$ and $g\in\Z^3$ we have
  $\tilde\rho_j(h)m_j(g)\tilde\rho_j(h)^{-1}=m_j(hgh^{-1})$.
\end{itemize}
\end{definition}

We should point out here that it is by far not always possible to
extend a choice of $(E_0,\theta)$ and $\{(E_j,A_j)\}$ to compatible
gluing data.  This will become clear from the discussion in
\fullref{sec:ex}.

Before we proceed to construct approximate $\rG_2$--instantons, we
introduce weighted H\"older norms.  It will become more transparent over
the course of the next two sections that these are well adapted to the
problem at hand. We define weight functions by
\begin{equation*}
  w_t(x) := t + r_t(x)
  \quad\text{and}\quad
  w_t(x,y) := \min\{w_t(x),w_t(y)\}.
\end{equation*}
For $t\in(0,T)$, a H\"older exponent $\alpha\in(0,1)$ and a weight
parameter $\beta\in\R$ we define
\begin{align*}
  [f]_{C^{0,\alpha}_{\beta,t}(U)}
  &:= \sup_{d(x,y) \leq w_t(x,y)}
        w_t(x,y)^{\alpha-\beta} \frac{|f(x)-f(y)|}{d(x,y)^\alpha}, \\[1ex]
  \|f\|_{L^{\infty}_{\beta,t}(U)}
  &:=\big\|w_t^{-\beta}f\big\|_{L^\infty(U)}, \\
\|f\|_{C^{k,\alpha}_{\beta,t}(U)} 
  &:= \sum_{j=0}^k \big\|\nabla^j f\big\|_{L^{\infty}_{\beta-j,t}(U)}
                + \big[\nabla^j f\big]_{C^{0,\alpha}_{\beta-j,t}(U)}.
\end{align*}
Here $f$ is a section of a vector bundle over $U\subset Y_t$ equipped
with an inner product and a compatible connection.  On tensor bundles
associated with $Y_t$ we use the metrics induced by $\tilde g_t$;
however, in view of \fullref{prop:joyce+}, we could
equivalently use those induced by $\phi_t=\tilde \phi_t+\rd \eta_t$.
We use parallel transport to compare the value of $f$ at different
points in $Y$.  If $U$ is not specified, then we take $U=Y_t$.  We
denote by $\smash{C^{k,\alpha}_{\beta,t}}$ the Banach space $C^{k,\alpha}$
equipped with the norm $\|\cdot\|_{C^{k,\alpha}_{\beta,t}}$

\begin{remark}
  For fixed $t\in(0,T)$ and $\beta\in\R$, the norms
  $\smash{\|\cdot\|_{C^{k,\alpha}_{\beta,t}}}$ and are
  $\smash{\|\cdot\|_{C^{k,\alpha}}}$ equivalent, but not uniformly so as $t>0$
  tends to zero.
\end{remark}

Note that, if
$\beta=\beta_1+\beta_2$, then
\begin{equation}\label{eq:m}
  \|f\cdot{}g\|_{C^{k,\alpha}_{\beta,t}} \leq
  \|f\|_{C^{k,\alpha}_{\beta_1,t}}\cdot\|g\|_{C^{k,\alpha}_{\beta_2,t}}.
\end{equation}
Also for $\beta>\gamma$ we have
\begin{equation}\label{eq:wi}
  \|f\|_{C^{k,\alpha}_{\beta,t}} \leq
  t^{\gamma-\beta}\|f\|_{C^{k,\alpha}_{\gamma,t}}.
\end{equation}

\begin{prop}\label{prop:gdc}\label{prop:ag}
  Let $\bg$ be gluing data compatible with $\br$.  Then there is a
  constant $c>0$ and for each $t\in(0,T)$ a $G$--bundle $E_t$ over
  $Y_t$ together with a connection $\tilde A_t$ satisfying
  \begin{equation}\label{eq:e0e}
    \big\|F_{\tilde A_t}\wedge\psi_t\big\|_{C^{0,\alpha}_{-2,t}}\leq
    c t^{1/2}.
  \end{equation}
  Moreover, the adjoint bundle $\frg_{E_t}$ associated with $E_t$
  satisfies
  \begin{equation}\label{eq:p1'}
    p_1(\frg_{E_t})=-\sum_j k_j\,\PD[S_j]
   \quad\text{with }\,
    k_j:=\frac{1}{8\pi^2}\int_{X_j} |F_{A_j}|^2
  \end{equation}
  and
  \begin{equation}\label{eq:w2'}
    \langle w_2(\frg_{E_t}),[\Sigma]\rangle =\langle
w_2(\frg_{E_j}),[\Sigma]\rangle
  \end{equation}
  for each $[\Sigma]\in H_2(X_j)^{H_j}\subset H_2(Y_t)$.
\end{prop}

\begin{proof}
  The choices of $\tilde\rho_j$ and $m_j$ define a lift of the action
  of $\Z^3\rtimes H_j$ on $\R^3\times X_j$ to the pullback of $E_j$ to
  $\R^3\times X_j$.  Passing to the quotient yields a $G$--bundle over
  $(T^3\times X_j)/H_j$ which we denote by $E_j$, by abuse of
  notation.  It follows from the compatibility conditions that the
  pullback of $A_j$ to $\R^3\times X_j$ passes to the quotient and
  induces a connection on $E_j$ which we denote by $A_j$, again by
  abuse of notation.
  
  Fix $t\in(0,T)$.  Recall that in \eqref{eq:Rt} we defined
  $R_{j,t}:=\tilde T_{j,t}\cap r_t^{-1}[\zeta/4,\zeta/2]$ with $\tilde
  T_{j,t}$ and $r_t$ as defined in \eqref{eq:Tjt} and \eqref{eq:rt},
  respectively.  By the compatibility conditions the monodromy of
  $A_j$ along $S_j$ on the fibre at infinity matches up with the
  monodromy of $\theta$ along $E_0|_{S_j}$.  Thus, via parallel
  transport the framing of $E_0$ at $x_j$ and the framing of $E_j$
  yield an identification of $E_0|_{R_{j,t}}$ with $E_j|_{R_{j,t}}$.
  Patching $E_0$ and the $E_j$ via this identification yields the
  bundle $E_t$.

  Under the identification of $E_0|_{R_{j,t}}$ with $E_j|_{R_{j,t}}$,
  we can write
  \begin{equation}\label{eq:aj}
    A_j=\theta+a_j \quad\text{with }\,
    \nabla^k a_j=t^{2+k}O(r_t^{-3-k}),
  \end{equation}
  because of \fullref{rmk:mcmp} and \fullref{prop:decay}.
  Fix a smooth
  non-increasing function $\chi\co [0,\zeta]\to[0,1]$ such that
  $\chi(s)=1$ for $s\leq\zeta/4$ and $\chi(s)=0$ for $s\geq\zeta/2$.
  Set $\chi_t:=\chi\circ r_t$.  After cutting off $A_j$ to
  $\theta+\chi_t\cdot a_j$ it can be matched with $\theta$ and we
  obtain the connection $\tilde A_t$ on the bundle $E_t$.

  To estimate $\smash{F_{\tilde A}}\wedge\psi_t$ note that on $Y_t \setminus
  \tilde T_t$ the connection $\tilde A_t$ is flat.  Thus we can focus
  our attention on $\tilde T_{j,t}$.  By the definition of $\tilde
  A_t$ we have
  \begin{equation*}
    F_{\tilde A_t}=\chi_tF_{A_j}
          +\rd\chi_t\wedge a_j
          +\frac{\chi_t^2-\chi_t}2[a_j\wedge a_j].
  \end{equation*}
  The last two terms in this expression are supported in $R_{j,t}$ and of
  order $t^2$ in $C^{0,\alpha}$ by \eqref{eq:aj}.  By
  \fullref{ex:asdg2} and \fullref{prop:joyce+} we have
  \begin{equation*}
    \big\|F_{A_j}\wedge\psi_t\big\|_{C^{0,\alpha}_{-2,t}(\tilde
      T_{j,t})}
    = \big\|F_{A_j}\wedge\big(\psi_t-\hat\psi_t\big)\big\|_{C^{0,\alpha}_{-2,t}(\tilde
      T_{j,t})}
    \leq ct^{1/2} \big\|F_{A_j}\big\|_{C^{0,\alpha}_{-2,t}}.
  \end{equation*}
  It follows from \fullref{prop:Adecay} and \fullref{rmk:mcmp} that
  \begin{equation*}
    \nabla^k F_{A_j}=t^{2+k}O(r_t^{-4-k}).
  \end{equation*}
  This implies that
$$\big\|F_{A_j}\big\|_{C^{0,\alpha}_{-4,t}(\tilde T_{j,t})} \leq ct^{2}$$
and, hence,
  $$\big\|F_{A_j}\big\|_{C^{0,\alpha}_{-2,t}(\tilde T_{j,t})} \leq c$$
by \eqref{eq:wi} with $c>0$ independent of $t\in(0,T)$.  Now, putting
  everything together yields~\eqref{eq:e0e}.%

  Let $\iota_{j,t}\co T^3\times
\pi_{\smash{j,t}}^{-1}(B_{\smash{\zeta}}^4/G_j) \to Y$
  be as in \fullref{rmk:topy}.  Then $\iota_{\smash{j,t}}^*\frg_{E_t}$ is
  isomorphic to the pullback of $\frg_{E_j}$ to $T^3\times
  \pi_{\smash{j,t}}^{-1}(B_{\smash{\zeta}}^4/G_j)$.  This implies \eqref{eq:w2'} by
  naturality of Stiefel--Whitney classes.  To compute
  $p_1(\frg_{E_t})$ we use Chern--Weil theory to represent it as
  $p_1(\frg_{E_t})=
\smash{-\frac{1}{8\pi^2}}\tr (F_{\smash{\tilde A_t}}\wedge
  F_{\smash{\tilde A_t}})$.  We can write this as $p_1(\frg_{E_t})=\sum_{j}
  p_j$, where $p_j$ are compactly supported $4$--forms on $\smash{\tilde
  T_{j,t}}$.  Recalling the definition of $[S_j]$ in \eqref{eq:Sj} and
  considering the behaviour of Poincar\'e duality with respect to
  coverings we see that in order to prove \eqref{eq:p1'} we have to
  show 
  \begin{equation*}
    \iota_{j,t}^*p_j=k_j\, \PD\big[T^3\times\{x\}\big] \in
    H^4_c\big(T^3\times \pi_{j,t}^{-1}\big(B_\zeta^4/G_j\big),\R\big).
  \end{equation*}
  From our construction of $\tilde A_t$ it follows that the form
  $\iota_{\smash{j,t}}^*p_j$ is the pullback of a compactly supported
  $4$--form on $X_j$ which we can write as $\smash{-\frac{1}{8\pi^2}}\tr
  (F_{\smash[t]{\tilde A_j}}\wedge F_{\smash{\tilde A_j}})$ where $\tilde A_j=A_j+\alpha$
  and, by slight abuse of notation, $\alpha=(1-\chi_t)a_j$.
  Consequently, $\iota_{\smash{j,t}}^*p_j$ is a multiple of
  $\PD[T^3\times\{x\}]$.  To see that the multiplicity is
  precisely $k_j$ we use the Chern--Simons $3$--form (see
  Donaldson--Kronheimer~\cite[Equation~(2.1.17)]{Donaldson1990}) to write
  \begin{equation*}
    \tr\big(F_{\tilde A_j}\wedge F_{\tilde A_j}\big)-\tr\big(F_{A_j}\wedge
F_{A_j}\big)
    =\rd\tr\big(\alpha\wedge\rd_{A_j}\alpha + \tfrac13
    \alpha\wedge[\alpha\wedge\alpha]\big).
  \end{equation*}
  By \fullref{prop:decay} the $1$--form $\alpha$ decays
  sufficiently fast to conclude from Stokes' theorem that
  \begin{equation*}
    -\frac{1}{8\pi^2} \int_{X_j} \tr \big(F_{\tilde A_j}\wedge F_{\tilde
      A_j}\big) = -\frac{1}{8\pi^2} \int_{X_j} \tr\big(F_{A_j}\wedge
F_{A_j}\big)
    = \int_{X_j} \frac{1}{8\pi^2}\big|F_{A_j}\big|^2=k_j.
  \end{equation*}
  This completes the proof.
\end{proof}

\begin{remark}
  If we identify all $Y_t$ with one fixed $Y$, then the isomorphism
  type of the bundles $E_t$ does not depend on $t\in(0,T)$.  We can
  therefore think of them as one fixed $G$--bundle $E$ over $Y$.
\end{remark}

\section[A model operator on R^3 x ALE]
{A model operator on $\R^3\times \text{ALE}$}
\label{sec:model}

In order to prove \fullref{thm:a} we need to find
$\xi_t\in\Omega^0(Y_t,\frg_{E_t})$ and
$a_t\in\Omega^1(Y_t,\frg_{E_t})$ such that
\begin{equation}\label{eq:y}
  *_t\big(F_{\tilde A_t+a_t}\wedge\psi_t\big)+\rd_{\tilde A_t}\xi_t=0
\end{equation}
for $t\in(0,T')$ provided $T'\in(0,T]$ is sufficiently small.  Here
$*_t$ denotes the Hodge $*$--operator associated with $\phi_t$.
Equation \eqref{eq:y} together with the Coulomb gauge condition
$\rd_{\tilde A_t}^*a_t=0$ can be written as
\begin{equation}\label{eq:x}
  L_{t} \ua_t + Q_t(\ua_t)
   + *_t\big(F_{\tilde A_t}\wedge\psi_t\big) = 0.
\end{equation}
Here we use the notation $\ua_t:=(\xi_t,a_t)$, the linear operator
$L_t:=L_{\tilde A_t}$ is defined as in \eqref{eq:la} with
$\psi=\psi_t:=*_t\phi_t$ and $Q_t$ is defined by
\begin{equation}\label{eq:qt}
  Q_t(\ua)
  := \tfrac12 *_t([a\wedge a]\wedge\psi_t) + [a,\xi].
\end{equation}
The key to solving \eqref{eq:x} is a good understanding of the
linearisation $L_t$.  In this section, we study a model
for $L_t$ on $r_t^{-1}([0,\zeta))$.

Let $X$ be an ALE space, let $A$ be a $G$--bundle over $X$ and let $A$
be a finite energy ASD instanton on $E$.  Fix an orthonormal triple
$(\delta^1,\delta^2,\delta^3)$ of constant $1$--forms on $\R^3$ and
denote by $(\omega_1,\omega_2,\omega_3)$ the triple of K\"ahler forms
associated with $X$.  Consider $\R^3\times X$ as a $\rG_2$--manifold
as in \fullref{ex:t3x}.  Denote by $p_{\R^3}\co \R^3\times X\to
\R^3$ and $p_X\co \R^3\times X\to X$ the projection onto the first and
second factor, respectively.  Slightly abusing notation, we denote the
respective pullbacks of $E$ and $A$ to $\R^3\times X$ via $p_X$ by $E$
and $A$ as well.  As in \eqref{eq:la} we define
$L_A\co\Omega^0(\R^3\times X,\frg_E)\oplus\Omega^1(\R^3\times
X,\frg_E) \to \Omega^0(\R^3\times X,\frg_E)\oplus\Omega^1(\R^3\times
X,\frg_E)$ by
\begin{equation*}
L_A=\begin{pmatrix}
0 & \rd_A^* \\
\rd_A & *(\psi\wedge\rd_{A})
\end{pmatrix}
\end{equation*}
with $\psi$ as in \eqref{eq:psi}.

\begin{prop}\label{prop:lala}
  If we identify $p_{\R^3}^*T^*\R^3$ with $p_X^*\Lambda^+T^*X$ via
  $\delta^1\mapsto\omega_1$, $\delta^2\mapsto\omega_2$,
  $\delta^3\mapsto-\omega_3$ and accordingly
  \begin{equation*}
    \Omega^0(\R^3\times X,\frg_E)\oplus\Omega^1(\R^3\times X,\frg_E)
    = \Omega^0\big(\R^3{\times}X,p_X^*\big[(\R{\oplus}\Lambda^+T^*X{\oplus}
    T^*X){\otimes}\frg_E\big]\big),
  \end{equation*}
  then the operator $L_A$ can be written as $L_A=F+D_A$, where
  \begin{equation*}
    F(\xi,\omega,a) = \sum_{i=1}^3 \big(-\langle\del_i\omega,
    \omega_i\rangle,
     \del_i \xi \cdot \omega_i, I_i\del_i a\big) \quad\text{and}\quad 
    D_A=\begin{pmatrix}
      0 & \delta_A \\
      \delta_A^* & 0
    \end{pmatrix}.
  \end{equation*}
  Here
  $\delta_A\co\Omega^1(X,\frg_E)\to\Omega^0(X,\frg_E)\oplus\Omega^+(X,\frg_E)$
  denotes the linear operator defined in \eqref{eq:deltaA}.  Moreover,
  \begin{equation}\label{eq:lala}
    L_A^*L_A = \Delta_{\R^3}  +
    \begin{pmatrix}
      \delta_A\delta_A^* & \\
      & \delta_A^*\delta_A
    \end{pmatrix}
  \end{equation}
  where $\Delta_{\R^3}=-\sum_{i=1}^3 \del_i^2$ and $\del_i$ denotes
  taking the derivative of a section of
  $p_X^*[(\R\oplus\Lambda^+T^*X\oplus T^*X)\otimes\frg_E)]$ in the
  direction of the $i^{\text{th}}$ coordinate on $\R^3$.
\end{prop}

\begin{proof}
  It is a straight-forward computation to verify that $L_A=F+D_A$.  It
  is also easy to see that $F^*F=\Delta_{\R^3}$ and that
  $F^*D_A+D_A^*F=0$.  This immediately implies \eqref{eq:lala}.
\end{proof}

To understand the properties of $L_A$ we work with weighted
H\"older norms.  We define weight functions by
\begin{equation*}
  w(x) := 1 + |\pi(p_{X}(x))| \quad\text{and}\quad
  w(x,y) := \min\{w(x),w(y)\}.
\end{equation*}
Here $\pi\co X\to\C^2/G$ denotes the resolution map associated with
the ALE space $X$.  For a H\"older exponent $\alpha\in(0,1)$ and a
weight parameter $\beta\in\R$ we define
\begin{align*}
  [f]_{C^{0,\alpha}_{\beta}(U)}
  &:= \sup_{d(x,y) \leq w(x,y)}
        w(x,y)^{\alpha-\beta} \frac{|f(x)-f(y)|}{d(x,y)^\alpha}, \\[1ex]
  \|f\|_{L^{\infty}_{\beta}(U)}
  &:=\big\|w^{-\beta}f\big\|_{L^\infty(U)}, \\
\|f\|_{C^{k,\alpha}_{\beta}(U)}
  &:= \sum_{j=0}^k \big\|\nabla^j f\big\|_{L^{\infty}_{\beta-j}(U)}
                + \big[\nabla^j f\big]_{C^{0,\alpha}_{\beta-j}(U)}.
\end{align*}
Here $f$ is a section of a vector bundle over $U\subset\R^3\times X$
equipped with an inner product and a compatible connection.  We use
parallel transport to compare the values of $f$ at different points.
If $U$ is not specified, then we take $U=Y_t$.  We denote by
$C^{k,\alpha}_{\smash{\beta}}$ the subspace of elements $f$ of the Banach space
$C^{k,\alpha}$ with $\smash{\|f\|_{C^{k,\alpha}_\beta}}<\infty$ and equip it
with the norm $\smash{\|\cdot\|_{C^{k,\alpha}_\beta}}$.

Under the assumptions of \fullref{sec:approx} and with $\bg$
denoting compatible gluing data suppose that $X=X_j$ and that $A=A_j$.
Define $\tilde \iota_{j,t}\co\R^3\times
\pi_{j,t}^{-1}(B_\zeta^4/G_j)\to\tilde T_{j,t}$ by
\begin{equation*}
  \tilde \iota_{j,t}(x,y):=[(tx,y)].
\end{equation*}  
For a parameter $\beta\in\R$ and
$\ua=(\xi,a)\in\Omega^0(Y_t,\frg_{E_t})\oplus\Omega^1(Y_t,\frg_{E_t})$
we define
\begin{equation}\label{eq:sbt}
  s_{\beta,t} (\xi,a)(x,y)
  :=t^{\beta-1}\big(t (\tilde\iota_{j,t})^*\xi,(\tilde\iota_{j,t})^* a\big).
\end{equation}

\begin{prop}\label{prop:ltcmp}
  There is a constant $c>0$ such that for $t\in(0,T)$
  \begin{gather*}
    \tfrac1c \|\ua\|_{C^{k,\alpha}_{\beta,t}(\tilde{T}_{j,t})} \leq 
    \|s_{\beta,t}\ua\|_{C^{k,\alpha}_{\beta}(\R^3\times
      \pi_{j,t}^{-1}(B_\zeta^4/G_j))} \leq c
    \|\ua\|_{C^{k,\alpha}_{\beta,t}(\tilde{T}_{j,t})},
    \\[2pt]
    \|L_t \ua- s_{\beta-1,t}^{-1}L_{A_j} s_{\beta,t}\ua\|_{C^{0,\alpha}_{\beta-1,t}(\tilde T_{j,t})}
    \leq ct^{1/2}\|\ua\|_{C^{1,\alpha}_{\beta,t}(\tilde T_{j,t})}.
  \end{gather*}
\end{prop}

\begin{proof}
  The map $\tilde \iota_{j,t}$ pulls back the metric on $\tilde
  T_{j,t}$ associated with $\hat\phi_t$, that is
  $g_{\smash{\hat\phi_t}}=g_{\R^3}\oplus t^2g_{X_j}$, to $t^2(g_{\R^3}\oplus
  g_{X_j})$.  This implies the first estimate in view of
  \fullref{rmk:mcmp}.  The second estimate is immediate from
  the construction of $\smash{\tilde A_t}$ and \fullref{prop:joyce+}.
\end{proof}

\begin{prop}\label{prop:lk}
  Let $\beta\in(-3,0)$. Then $\ua\in C^{1,\alpha}_{\smash{\beta}}$ is in the
  kernel of $L_A\co C^{1,\alpha}_{\smash{\beta}}\to C^{0,\alpha}_{\smash{\beta-1}}$ if
  and only if it is given by the pullback of an element of the $L^2$
  kernel of $\delta_A$ to $\R^3\times X$.
\end{prop}

The proof of \fullref{prop:lk} relies on the following lemma
which we will prove in the \hyperlink{appendixa}{Appendix}.

\begin{definition}
  A Riemannian manifold $X$ is said to be of \emph{bounded geometry}
  if it is complete, its Riemann curvature tensor is bounded from
  above and its injectivity radius is bounded from below.  A vector
  bundle over $X$ is said to be of \emph{bounded geometry} if it has
  trivialisations over balls of a fixed radius such that the
  transitions functions and all of their derivatives are uniformly
  bounded.  We say that a complete oriented Riemannian manifold $X$
  has \emph{subexponential volume growth} if for each $x\in X$ the
  function $r\mapsto\vol(B_r(x))$ grows subexponentially, that is,
  $\vol(B_r(x))=o(\exp(c r))$ as $r\to \infty$ for every $c>0$.
%  \vadjust{\goodbreak}
\end{definition}

\begin{lemma}\label{lem:liouville}
  Let $E$ be a vector bundle of bounded geometry over a Riemannian
  manifold $X$ of bounded geometry and with subexponential volume
  growth, and suppose that $D\co C^\infty(X,E)\to{}C^\infty(X,E)$ is a
  uniformly elliptic operator of second order whose coefficients and
  their first derivatives are uniformly bounded, that is non-negative,
  such that $\langle Da,a\rangle\geq 0$ for all $a \in W^{2,2}(X,E)$, and formally
  self-adjoint.  If $a \in C^\infty(\R^n\times X,E)$ satisfies
  \begin{equation*}
    (\Delta_{\R^n}+D)a=0
  \end{equation*}
  and $\|a\|_{L^\infty}$ is finite, then $a$ is constant in the
  $\R^n$--direction, that is $a(x,y)=a(y)$.  Here, by slight abuse of
  notation, we denote the pullback of $E$ to $\R^n\times X$ by $E$ as
  well.
\end{lemma}

\begin{proof}[Proof of \fullref{prop:lk}]
  Suppose $\ua\in C^{1,\alpha}_{\smash{\beta}}$ satisfies $L_A\ua=0$.  Then
  $\ua$ is smooth by elliptic regularity and satisfies
  $\smash{L_A^*L_A\ua}=0$.  By \fullref{def:ale} and by
  \fullref{prop:Adecay} both $\R^3\times X$ and $\frg_E$ have
  bounded geometry.  Moreover, by \fullref{prop:lala},
  $L_A^*L_A=\Delta_{\R^3}+D_A^*D_A$ and $D_A^*D_A$ is non-negative,
  self-adjoint, uniformly elliptic of second order and its
  coefficients and their first derivatives are uniformly bounded as
  can be seen from \fullref{prop:Adecay}.  Therefore, we can
  apply \fullref{lem:liouville} to conclude that $\ua$ is invariant
  under translations in the $\R^3$--direction and, hence, by
  Propositions \ref{prop:decay} and \ref{prop:lala} must be the
  pullback of an element in the $L^2$ kernel of $\delta_A$.
\end{proof}

\begin{prop}\label{prop:mse}
  For $\beta\in\R$ there is a constant $c>0$ such that
  \begin{equation*}
    \|\ua\|_{C^{1,\alpha}_{\beta}}
    \leq c\Big(\|L_A\ua\|_{C^{0,\alpha}_{\beta-1}}
       + \|\ua\|_{L^{\infty}_{\beta}}\Big).
  \end{equation*}
\end{prop}

\begin{proof}
  This is a standard result; see \fullref{rmk:nw}.

  The desired estimate is local in the sense that is enough to prove
  estimates of the form
  \begin{equation*}
    \|\ua\|_{C^{1,\alpha}_{\beta}(U_i)} \leq c\Big(\|L_A
    \ua\|_{C^{0,\alpha}_{\beta-1}} + \|\ua\|_{L^{\infty}_{\beta}}\Big)
  \end{equation*}
  with $c>0$ independent of $i$, where $\{U_i\}$ is a suitable open
  cover of $\R^3\times X$.

  Fix $R>0$ suitably large and set $U_0:=\{(x,y)\in\R^3\times{}X :
  |\pi(x)|\leq R\}$. Then there clearly is a constant $c>0$ such that
  the above estimate holds for $U_i=U_0$.  Pick a sequence
  $(x_i,y_i)\in\R^3\times{}X$ such that $r_i:=|\pi(y_i)|\geq R$ and
  the balls $U_i:=B_{r_i/8}(x_i,y_i)$ cover the complement of $U_0$.
  On $U_i$, we have a Schauder estimate of the form
  \begin{multline*}
    \|\underline{a}\|_{L^\infty(U_i)} +
    r_i^{\alpha}[\underline{a}]_{C^{0,\alpha}(U_i)} + r_i\big\|\nabla_A
    \underline{a}\big\|_{L^\infty(U_i)} +
    r_i^{1+\alpha}\big[\nabla_A \underline{a}\big]_{C^{0,\alpha}(U_i)} \\
    \leq c\big(r_i\big\|L_A\underline a\big\|_{L^\infty(V_i)} +
      r_i^{1+\alpha}\big[L_A\underline{a}\big]_{C^{0,\alpha}(V_i)} +
      \|\underline{a}\|_{L^\infty(V_i)}\big)
  \end{multline*}
  where $V_i=B_{r_i/4}(x_i,y_i)$ and $\underline{a}=(\xi,a)$.  By
  arguing as in Propositions \ref{prop:joyce+} and \ref{prop:decay}
  one shows that the constant $c>0$ can be chosen to work for all $i$
  simultaneously.  Since on $V_i$ we have $\frac12 r_i \leq w \leq
  2r_i$, multiplying the above Schauder estimate by $r_i^{\smash{-\beta}}$
  yields the desired local estimate.
  \vadjust{\goodbreak}
\end{proof}

\section[Deforming to genuine G_2--instantons]
{Deforming to genuine $\rG\sb2$--instantons}
\label{sec:deform}

We continue with the assumptions of \fullref{sec:approx} and we
suppose that the connection $\tilde A_t$ on $G$--bundle $E_t$ over
$Y_t$ was constructed using \fullref{prop:gdc} from a choice
of compatible gluing data $\bg$.  In this section we will prove the
following result which will complete the proof of \fullref{thm:a}.

\begin{prop}\label{prop:a}
  Suppose that $\theta$ is acyclic and that each $A_j$ is
  infinitesimally rigid.  Then there are constants $T'\in(0,T]$ and
  $c>0$ as well as, for each $t\in(0,T')$, $\ua_t=(\xi_t,a_t)\in
  \Omega^0(Y_t,\frg_{E_t})\oplus \Omega^1(Y_t,\frg_{E_t})$ such that
  \begin{equation}\label{eq:i}
    *_t\big(F_{\tilde A_t+a_t}\wedge\psi_t\big)+\rd_{\tilde A_t}\xi_t=0
  \end{equation}
  and $\smash{\|\ua_t\|_{C^{1,\alpha}_{-1,t}}}\leq ct^{1/2}$.  Moreover, the
  $\rG_2$--instanton $A_t:=\tilde A_t+a_t$ is acyclic.
\end{prop}

As discussed in \fullref{sec:model} it is crucial to understand
the properties of the linear operator $L_{t}$.  The key to proving
\fullref{prop:a} is the following result.
\begin{prop}\label{prop:key}
  Given $\beta\in(-3,0)$ there are constants $T'\in(0,T]$ and $c>0$
  such that for $t\in(0,T')$ we have
  \begin{equation*}
    \|\ua\|_{C^{1,\alpha}_{\beta,t}} \leq c
    \|L_t\ua\|_{C^{0,\alpha}_{\beta-1,t}}.
  \end{equation*}
\end{prop}

Before we move on to prove this, let us quickly show how it is used to
establish \fullref{prop:a}.  Recall the following elementary
consequence of Banach's fixed point theorem.

\begin{lemma}[Donaldson--Kronheimer~{\cite[Lemma~7.2.23]{Donaldson1990}}]
\label{lem:CM}
  Let $X$ be a Banach space and let $T \co X\to X$ be a smooth map
  with $T(0)=0$.  Suppose there is a constant $c>0$ such that
  \begin{equation*}
    \|Tx-Ty\|\leq c(\|x\|+\|y\|)\|x-y\|.
  \end{equation*}
  Then if $y\in X$ satisfies $\|y\|\leq
  \frac{1}{10c}$, there exists a unique $x\in X$ with $\|x\|\leq
  \frac{1}{5c}$ solving
  \begin{equation*}
    x+Tx=y.
  \end{equation*}
  Moreover, this $x\in X$ satisfies $\|x\|\leq 2\|y\|$.
\end{lemma}

\begin{proof}[Proof of \fullref{prop:a} assuming
  \fullref{prop:key}]
  By \fullref{prop:key} the operator $L_t\co
  \smash{C^{1,\alpha}_{-1,t}}\to \smash{C^{0,\alpha}_{-2,t}}$ is
  injective and has closed range.  Therefore its cokernel is isomorphic
  to the kernel of the dual operator $L_t^*$.  By elliptic regularity
  any element in the kernel of $L_t^*$ is smooth and thus, since $L_t$
  is formally self-adjoint, an element in the kernel of $L_t$, which
  is trivial.  This shows that $L_t$ is invertible.  Denote its inverse
  by $R_t\co \smash{C^{0,\alpha}_{-2,t}}\to \smash{C^{1,\alpha}_{-1,t}}$.

  If we set $\ua_t:=R_t\ub_t$, then \eqref{eq:i} becomes
  \begin{equation}\label{eq:z}
    \ub_t+ Q_t(R_t\ub_t) = -*_t\big(F_{\tilde A_t}\wedge\psi_t\big).
  \end{equation}
  It follows from \fullref{prop:key} and \eqref{eq:m} that
  \begin{equation*}
    \big\|Q_t(R_t\ub_1)-Q_t(R_t\ub_2)\big\|_{C^{0,\alpha}_{-2,t}} 
    \leq c \Big(\|\ub_1\|_{C^{0,\alpha}_{-2,t}} +
      \|\ub_2\|_{C^{0,\alpha}_{-2,t}}\Big)
    \|\ub_1-\ub_2\|_{C^{0,\alpha}_{-2,t}}
  \end{equation*}
  with a constant $c>0$ independent of $t\in(0,T)$.  Since by
  \fullref{prop:ag}
  \begin{equation*}
    \big\|F_{\tilde A_t}\wedge\psi_t\big\|_{C^{0,\alpha}_{-2,t}}\leq ct^{1/2},
  \end{equation*}
  \fullref{lem:CM} provides us with, for each $t\in(0,T')$, a
  solution $\ub_t$ of \eqref{eq:z} satisfying
  $\smash{\|\ub_t\|_{C^{0,\alpha}_{-2,t}}}\leq ct^{1/2}$ provided $T'\in(0,T]$
  was chosen sufficiently small.  Then
$$\ua_t=(\xi_t,a_t)=R_t\ub_t\in C^{1,\alpha}_{-1,t}$$
is the desired solution of \eqref{eq:i} and
  satisfies $\smash{\|\ua_t\|_{C^{1,\alpha}_{-1,t}}}\leq ct^{1/2}$.

  It follows from elliptic regularity that $a_t$ and thus $A_t:=\tilde
  A_t + a_t$ is smooth.  To see that $A_t$ is acyclic, that is,
  $L_{A_t}$ is injective, note that
  $\|R_tL_{\smash{A_t}}-\id\|_{\smash{C^{1,\alpha}_{-1,t}}}\leq ct^{1/2}$ and thus
  $L_{A_t}$ is invertible for $t\in(0,T')$ provided $T'\in(0,T]$ was chosen
  sufficiently small.
\end{proof}

Before embarking on the proof of \fullref{prop:key}, it will
be helpful to make a few observations.  On $Y_t\setminus\tilde T_{t}$
the operators $L_t$ and $L_{\theta}$ agree.  For fixed $\epsilon>0$,
the norms
$\|\cdot\|_{\smash{C^{k,\alpha}_{\beta,t}(r_t^{-1}[\epsilon,\infty))}}$ are
uniformly equivalent to the corresponding unweighted H\"older norms.
Moreover, the restriction of $L_t$ to $r_t^{-1}[\epsilon,\infty)$
becomes arbitrarily close to $L_\theta$ restricted to $\{x\in
Y_0:d(x,S)>\epsilon\}$ as $t$ goes to zero.  These observations and
standard Schauder estimates combined with Propositions
\ref{prop:ltcmp} and \ref{prop:mse} yield the following Schauder
estimate.

\begin{prop}\label{prop:se}
  Given $\beta\in\R$ there is a constant $c>0$ such that for all
  $t\in(0,T)$ we have
  \begin{equation*}
    \|\ua\|_{C^{1,\alpha}_{\beta,t}}
    \leq c \big(\|L_t\ua\|_{C^{0,\alpha}_{\beta-1,t}} +
    \|\ua\|_{L^{\infty}_{\beta,t}}\big).
  \end{equation*}
\end{prop}

This reduces the proof of \fullref{prop:key} to the following
statement.

\begin{prop}\label{prop:2ndterm}
  Given $\beta\in(-3,0)$ there are constants $T'\in(0,T)$ and $c>0$
  such that for all $t\in(0,T')$ the following holds:
  \begin{equation*}
    \|\ua\|_{L^\infty_{\beta,t}}
    \leq c\|L_t\ua\|_{C^{0,\alpha}_{\beta-1,t}}.
  \end{equation*}
\end{prop}

\begin{proof}
  Suppose not.  Then there exists a sequence $(\ua_i)$ and a
  null-sequence $(t_i)$ such that
  \begin{equation*}
    \|\ua_i\|_{L^{\infty}_{\beta,t_i}}=1
    \quad\text{and}\quad
     \|L_{t_i}\ua_i\|_{C^{0,\alpha}_{\beta-1,t_i}} \leq \frac1i.
  \end{equation*}
  Hence, by \fullref{prop:se}, we have
  \begin{equation}\label{eq:hb}
    \|\ua_i\|_{C^{1,\alpha}_{\beta,t_i}}
    \leq 2c.
  \end{equation}
  Pick $x_i \in Y_{t_i}$ such that
  \begin{equation*}
    w_{t_i}(x_i)^{-\beta}|\ua_i(x_i)|=1.
  \end{equation*}
  After passing to a subsequence we can assume that one of the
  following three cases occurs.  We will rule out all of them, thus
  proving the proposition.

  \setcounter{case}{0}
  \begin{case}
    The sequence $(x_i)$ accumulates on the regular part of $Y_0$: $\lim
    r_{t_i}(x_i)>0$.
  \end{case}

  Let $K$ be a compact subset of $Y_0\setminus S$.  We can view $K$ as
  a subset of $Y_t$.  As $t$ goes to zero, the metric on $K$ induced
  from the metric on $Y_t$ converges to the metric on $Y_0$, similarly
  we can identify $E_0|_K$ with $E_t|_K$ and via this identification
  $\tilde A_t$ converges to $\theta$ on $K$.  By \eqref{eq:hb} the
  sequence $(\ua_i|_K)$ is uniformly bounded in $C^{1,\alpha}$.  We
  can thus extract a convergent subsequence using Arzel\`a--Ascoli.
  Using a diagonal sequence argument over a sequence of compact sets
  $(K_i)$ exhausting $Y_0\setminus S$ we can pass to a further
  subsequence which converges in $\smash{C^{1,\alpha/2}_\loc}$ to a limit $\ua
  \in \Omega^0(Y_0\setminus S,\frg_{E_0})\oplus\Omega^1(Y_0\setminus
  S,\frg_{E_0})$. This limit satisfies
  \begin{equation}\label{eq:bu1}
    |\ua|< c \cdot d(\,\cdot\,,S)^{\beta}
  \end{equation}
  as well as
  \begin{equation*}
    L_\theta\ua=0.
  \end{equation*}
  Since $\beta>-3$, it follows from \eqref{eq:bu1} that $\ua$
  satisfies $L_\theta\ua=0$ in the sense of distributions on all of
  $Y_0$ and, therefore, is smooth by elliptic regularity.  Because
  $\theta$ is assumed to be acyclic, $\ua$ must be zero.
  However, by passing to a further subsequence we can arrange that
  $(x_i)$ converges to some point $x \in Y_0\setminus S$. At this
  point we have $|\ua|(x)=d(x,S)^\beta \neq 0$.  This is a
  contradiction.

  \begin{case}
    The sequence $(x_i)$ accumulates on one of the ALE spaces: $\lim
    r_{t_i}(x_i)/t_i<\infty$.
  \end{case}

  There is no loss in assuming that each $x_i$ lies in $\tilde
  T_{j,t_i}$ for some fixed $j$.  With $s_{\beta,t_i}$ as in
  \eqref{eq:sbt} we define $\tilde\ua_i:=s_{\beta,t_i}\ua_i$ and
  denote by $\tilde x_i$ a lift of $x_i$ to $\R^3\times
  \pi_{j,t}^{-1}(B_\zeta^4/G_j)$.  This rescaled sequence satisfies,
  in the notation of \fullref{sec:model},
    $$\left\|\tilde \ua_i\right\|_{C^{1,\alpha}_{\beta}}\leq 4c \quad\text{and}\quad
    (1+|\pi_j(\tilde x_i)|)^{-\beta}
    |\tilde\ua(\tilde x_i)|
    \geq \tfrac12$$
  as well as
\begin{equation}
    \label{eq:lai}
    \|L_{A_j}\tilde\ua_i\|_{C^{0,\alpha}_{\beta-1}} \leq
    2/i.
\end{equation}
  Arguing as in the previous case, we can extract a subsequence of
  $(\tilde\ua_i)$ which converges to a limit $\tilde \ua \in
  C^{1,\alpha/2}_{\beta}$ in $C^{1,\alpha/2}_\loc$ on $\R^3\times X_j$.
  It follows from \eqref{eq:lai} that $\tilde\ua$ satisfies
  \begin{equation*}
    L_{A_j}\tilde\ua=0.
  \end{equation*}
  By \fullref{prop:lk}, $\tilde\ua$ must be zero since
  $\beta\in(-3,0)$ and $A_j$ is infinitesimally rigid.  However, by
  translation we can arrange that the $\R^3$--component of $\tilde
  x_i$ is zero and thus we can view $\tilde x_i$ as a point in $X_j$.
  Then the condition $\lim d_{t_i}(x_i)/t_i<\infty$ translates to
  $\lim |\pi_j(\tilde x_i)|<\infty$.  Therefore, we can assume without
  loss of generality that $\tilde x_i$ converges to some point $\tilde
  x\in X_j$.  But then $|\tilde{\ua}(\tilde{x})| \geq
  \frac12(1+|\pi_j(\tilde x)|)^\beta > 0$, which contradicts
  $\tilde\ua=0$.

  \begin{case}
    The sequence $(x_i)$ accumulates on one of the necks: $\lim
    r_{t_i}(x_i)=0$ and $\lim r_{t_i}(x_i)/t_i=\infty$.
  \end{case}

  As in the previous case, we rescale to obtain $(\tilde \ua_i)$ and
  $(\tilde x_i)$, and we arrange it so that the $\R^3$--component of
  $\tilde x_i$ is zero.  Since $\lim d_{t_i}(x_i)/t_i=\infty$, we have
  $\lim |\pi_j(\tilde x_i)|=\infty$.  Fix a sequence $(R_i)$ tending
  to infinity such that $\epsilon_i:=R_i/|\pi_j(\tilde x_i)|$ goes to
  zero.  Using $\pi_j\co X\to \C^2/G$, we can think of the sets
  $\R^3\times(\C^2\setminus B_{R_i}^4)/G_j$ as subsets of
  $\R^3\times X_j$.  Restricting to these sets and rescaling
  everything by $1/|\pi_j(\tilde x_i)|$ we obtain, without changing
  notation, $\tilde\ua_i \in \Omega^0\(\R^3\times (\C^2\setminus
  B_{\epsilon_i}^4)/G_j\) \oplus \Omega^1\(\R^3\times (\C^2\setminus
  B_{\epsilon_i}^4)/G_j\)$ and $\tilde{x}_i\in\C^2\setminus
  B_{\epsilon_i}^4$ satisfying
    $$\|\tilde\ua_i\|_{C^{1,\alpha}_{\beta}}\leq 8c \quad\text{and}\quad
    |\tilde{x}_j|^{-\beta}|\tilde\ua_i(\tilde x_i)|
    \geq \tfrac14$$
  as well as
    $$\|L\tilde\ua_i\|_{C^{0,\alpha}_{\beta-1}} \leq 4/i.$$
  Here the norms $\|\cdot\|_{\smash{C^{k,\alpha}_{\beta}}}$ are defined like
  those in \fullref{sec:model} except with the weight function
  now defined by $w(x,y):=|y|$ for $(x,y)\in\R^3\times \C^2/G_j$.  The
  operator $L$ is defined by
  \begin{equation*}
    L(\xi,a):=(\rd^*a, \rd \xi+*(\psi_0\wedge\rd a))
  \end{equation*}
  with
  $\psi_0:=\frac12 \omega_1\wedge\omega_1
       +\delta^2\wedge\delta^3\wedge\omega_1
       +\delta^3\wedge\delta^1\wedge\omega_2
       -\delta^1\wedge\delta^2\wedge\omega_3$
  and
  $\omega_i\in\Omega^2(\C^2)$ as in \fullref{sec:kummer}.

  As before, we can extract a subsequence converging in
  $C^{1,\alpha/2}_\loc$ to a limit
$$\tilde\ua\in \Omega^0(\R^3\times
  \big(\C^2\setminus\{0\}\big)/G_j) \oplus \Omega^1(\R^3\times
  \big(\C^2\setminus\{0\}\big)/G_j)$$ satisfying
  \begin{equation}\label{eq:bu2}
    |\tilde\ua|< c w^{\beta}
  \end{equation}
  as well as
  \begin{equation*}
    L\tilde\ua=0.
  \end{equation*}
  Since $\beta>-3$, it follows from \eqref{eq:bu2} that $\tilde \ua$
  satisfies $L\tilde\ua=0$ in the sense of distributions on all of
  $\R^3\times \C^2/G_j$ and therefore $\tilde\ua$ is smooth by
  elliptic regularity.  It also follows from \eqref{eq:bu2} that both
  $\tilde\ua$ and $\nabla\tilde\ua$ are uniformly bounded: This is
  clear outside a tubular neighbourhood of $\R^3\times\{0\}$.  If
  $B_1$ is a ball of radius one centred at some point in
  $\R^3\times\{0\}$, then \eqref{eq:bu2} gives a uniform bound on
  $\|\tilde \ua\|_{L^p(B_1)}$, for some fixed $p\in(1,\infty)$.  Using
  elliptic estimates this yields a uniform $W^{k,p}$ estimate on the
  ball of radius one-half; hence, using Sobolev embedding, uniform
  bounds on $\tilde\ua$ and $\nabla\tilde\ua$.  Because
  $L^*L=\Delta_{\R^3}+\Delta_{\C^2}$, if follows from \fullref{lem:liouville} that $\tilde\ua$ is invariant under translations
  in the $\R^3$--direction.  Thus we can think of the components of
  $\tilde\ua$ as harmonic functions on $\C^2$.  Since $\beta<0$, they
  decay to zero at infinity and thus vanish identically.  However, we
  know that $|\tilde x_i|=1$ and thus a subsequence of $(\tilde x_i)$
  converges to a point $\tilde x \in\C^2/G_j$ with $|\tilde x|=1$ at
  which $|\tilde\ua|(\tilde x)\geq \frac14$, contradicting
  $\tilde\ua=0$.
\end{proof}

\section{Examples with $G=\SO(3)$}
\label{sec:ex}

We will now explain how to use \fullref{thm:a} to construct a few
concrete examples of $\rG_2$--instantons on the $\rG_2$--manifolds from
\cite[Sections~12.3 and~12.4]{Joyce2000}.  The flat $\rG_2$--structure
$\phi_0$ on $T^7$ given by \eqref{eq:phi0} is preserved by
$\alpha,\beta,\gamma\in\Diff(T^7)$ defined by
\begin{align*}
  \alpha(x_1,\ldots,x_7)&:=
  \big(x_1,x_2,x_3,-x_4,-x_5,-x_6,-x_7\big), \\
  \beta(x_1,\ldots,x_7)&:=
  \big(x_1,-x_2,-x_3,x_4,x_5,\tfrac12-x_6,-x_7\big),  \\
\gamma(x_1,\ldots,x_7)&:=
  \big(-x_1,x_2,-x_3,x_4,-x_5,x_6,\tfrac12-x_7\big).
\end{align*}
It is easy to see that $\Gamma:=\langle\alpha,\beta,\gamma\rangle\iso\Z_2^3$.

To understand the singular set $S$ of $T^7/\Gamma$ note that the only
elements of $\Gamma$ having fixed points are $\alpha$, $\beta$ and
$\gamma$.  The fixed point set of each of these elements consists of
$16$ copies of $T^3$.  The group $\langle\beta,\gamma\rangle$ acts freely on the
set of $T^3$ fixed by $\alpha$ and $\langle\alpha,\gamma\rangle$ acts freely on
the set of $T^3$ fixed by $\beta$, while
$\alpha\beta\in\langle\alpha,\beta\rangle$ acts trivially on the set of $T^3$
fixed by $\gamma$.  It follows that $S$ consists of $8$ copies of
$T^3$ coming from the fixed points of $\alpha$ and $\beta$ and $8$
copies of $T^3/\Z_2$.  Near the copies of $T^3$ the singular set is
modelled on $T^3\times \C^2/\Z_2$ while near the copies of $T^3/\Z_2$
it is modelled on $(T^3\times \C^2/\Z_2)/\Z_2$ where the action of
$\Z_2$ on $T^3\times \C^2/\Z_2$ is given by
\begin{equation*}
  (x_1,x_2,x_3,\pm(z_1,z_2))\mapsto\big(x_1,x_2,x_3+\tfrac12,\pm(z_1,-z_2)\big).
\end{equation*}
The 8 copies of $T^3$ can be desingularised by any choice of 8 ALE
spaces asymptotic to $\C^2/\Z_2$.  To desingularise the copies of
$T^3/\Z_2$ we need to chose ALE spaces which admit an isometric action
of $\Z_2$ asymptotic to the action $\Z_2$ on $\C^2/\Z_2$ given by
$\pm(z_1,z_2)\mapsto\pm(z_1,-z_2)$.  Two possible choices are the
resolution of $\C^2/\Z_2$ or a smoothing of $\C^2/\Z_2$.
See Joyce~\cite[pages~313--314]{Joyce2000} for details.

We construct our examples on desingularisations of quotients of
$T^7/\Gamma$.  To this end we define
$\sigma_1,\sigma_2,\sigma_3\in\Diff(T^7)$ by
\begin{align*}
  \sigma_1(x_1,\ldots,x_7)&:=
  \big(x_1,x_2,\tfrac12+x_3,\tfrac12+x_4,\tfrac12+x_5,x_6,x_7\big),\\
  \sigma_2(x_1,\ldots,x_7)&:=
  \big(x_1,\tfrac12+x_2,x_3,\tfrac12+x_4,x_5,x_6,x_7\big), \\
  \sigma_3(x_1,\ldots,x_7)&:=
  \big(\tfrac12+x_1,x_2,x_3,x_4,\tfrac12+x_5,\tfrac12+x_6,x_7\big).
\end{align*}
The elements $\sigma_j$ commute with all elements of $\Gamma$ and thus
act on $T^7/\Gamma$.  Moreover, this action is free.

\begin{example}
  Let $A:=\langle\sigma_2,\sigma_3\rangle$.  By analysing how $A$ acts on the
  singular set of $T^7/\Gamma$ one can see that the singular set of
  $Y_0:=T^7/(\Gamma \times A)$ consists of one copy of $T^3$, denoted
  by $S_1$, and $6$ copies of $T^3/\Z_2$, denoted by $S_2, \ldots,
  S_7$.  $S_1$ has a neighbourhood modelled on $T^3\times \C^2/\Z_2$,
  while $S_2,\ldots,S_6$ have neighbourhoods modelled on
  $(T^3\times\C^2/\Z_2)/\Z_2$ where $\Z_2$ acts by
  $\pm(z_1,z_2)\mapsto\pm(z_1,-z_2)$ on $\C^2/\Z_2$.  As before, $S_1$
  can be desingularised by any choice of an ALE space asymptotic to
  $\C^2/\Z_2$.  $S_2,\ldots,S_6$ can be desingularised by the
  resolution of $\C^2/\Z_2$ or a smoothing of $\C^2/\Z_2$.

  To compute the orbifold fundamental group $\pi_1(Y_0)$, note that it
  is isomorphic to the fundamental group $\pi_1(Y_0\setminus S)$ of
  the regular part of $Y_0$.  Denote by $p\co \R^7\to Y_0$ the canonical
  projection.  Then $p\co p^{-1}(Y_0\setminus S)\to Y_0\setminus S$ is
  a universal cover.  Up to conjugation we can therefore identify
  $\pi_1(Y_0)$ with the group of deck transformations
  \begin{equation*}
    \pi_1(Y_0)=\<\alpha,\beta,\gamma,\sigma_2,\sigma_3,
                \tau_1,\ldots,\tau_7\>\subset\Aff(7)=\GL(7)\ltimes \R^7.
  \end{equation*}
  Here we think of $\alpha,\beta,\gamma,\sigma_2,\sigma_3$ as elements
  of $\Aff(7)$ defined by the formulae above and $\tau_i$ translates
  the $i^{\rm th}$ coordinate of $\R^7$ by one.  The group $\pi_1(Y_0)$ is
  a non-split extension
  \begin{equation*}
    0\to\Z^7\to\pi_1(Y_0)\to\Gamma\times{}A\to 0.
  \end{equation*}
  To work out the orbifold fundamental group $\pi_1(T_j)$ of $T_j$,
  again up to conjugation, one simply has to understand the subgroup
  of deck transformations preserving a fixed component of
  $p^{-1}(T_j)\subset{}p^{-1}(Y_0\setminus S)$.  In this way one can
  compute
  \begin{align*}
    \pi_1(T_1)&=\<\alpha,\tau_1,\tau_2,\tau_3\>, \\
    \pi_1(T_2)&=\<\beta,\sigma_3\alpha,\tau_1,\tau_4,\tau_5\>, &
    \pi_1(T_3)&=\<\tau_3\beta,\sigma_3\alpha,\tau_1,\tau_4,\tau_5\>, \\
    \pi_1(T_4)&=\<\gamma,\alpha\beta,\sigma_2,\tau_4,\tau_6\>, &
    \pi_1(T_5)&=\<\tau_3\gamma,\tau_3\alpha\beta,\sigma_2,\tau_4,\tau_6\>, \\
    \pi_1(T_6)&=\<\tau_5\gamma,\tau_5\alpha\beta,\sigma_2,\tau_4,\tau_6\>, &
    \pi_1(T_7)&=\<\tau_3\tau_5\gamma,\tau_3\tau_5\alpha\beta,\sigma_2,\tau_4,\tau_6\>.
  \end{align*}
  Here $\tau_2$ does not appear explicitly in $\pi_1(T_j)$, for
  $j=4,\ldots,7$, because $\sigma_2^2=\tau_2\tau_4$.

  Denote by $V:=\<a,b,c \,|\, a^2=b^2=c^2=1, ab=c\>\iso\Z_2^2$ the Klein
  four-group. $V$ can be thought of as a subgroup of $\SO(3)$:
  $a=\diag(1,-1,-1)$, $b=\diag(-1,1,-1)$ and $c=\diag(-1,-1,1)$.  We
  define $\rho\co \pi_1(Y_0) \to V \subset \SO(3)$ by
  \begin{alignat*}{3}
    \beta,\gamma,\tau_1,\ldots,\tau_7 &\mapsto 1, &\qquad
    \alpha &\mapsto a, \\
    \sigma_2 &\mapsto a, & \qquad
    \sigma_3 &\mapsto b.
  \end{alignat*}
  To see that the flat connection $\theta$ induced by $\rho$ is
  acyclic we use the following observation.

  \begin{prop}\label{prop:fr}
    A flat connection $\theta$ on a $G$--bundle $E_0$ over a flat
    $\rG_2$--orbifold $Y_0$ corresponding to a representation
    $\rho\co\pi_1(Y_0)\to\rG$ is acyclic if and
    only if the induced representation of $\pi_1(Y_0)$ on
    $\frg\oplus(\R^7\otimes\frg)$ has no non-zero fixed vectors.
  \end{prop}

  \begin{proof}
    Since $Y_0$ is flat as a Riemannian orbifold and $\theta$ is a
    flat connection
    \begin{equation*}
      L_\theta^*L_\theta=\nabla_\theta^*\nabla_\theta.
    \end{equation*}
    Therefore, all elements in the kernel of $L_\theta$ are actually
    parallel sections of the bundle
    $\frg_{E_0}\oplus(T^*Y_0\otimes\frg_{E_0})$ and these are in
    one-to-one correspondence with fixed vectors of the representation
    of $\pi_1(Y_0)$ on $\frg\oplus(\R^7\otimes\frg)$.
  \end{proof}

  The elements $\sigma_2$ and $\sigma_3$ act trivially on $\R^7$ and
  their action on $\so(3)$ has no common non-zero fixed vectors.
  Therefore the action of $\pi_1(Y_0)$ on
  $\frg\oplus(\R^7\otimes\frg)$ has no non-zero fixed vector and
  thus $\theta$ is acyclic.

  The monodromy representation $\mu_j|_{G_j}\co G_j=\Z_2\to\SO(3)$
  associated with the flat connection $\theta$ is non-trivial only for
  $j=1$.  Let $A_1:=A_{0,1}$ be the infinitesimally rigid ASD
  instanton on $E_1:=E_{0,1}$ given in \fullref{prop:rigid}.
  For $j=2,\ldots,6$ we choose $A_j$ to be the product connection on
  the trivial $\SO(3)$--bundle $E_j$.  We take $m_1$ and
  $\tilde\rho_1$ to be trivial.  For $j=2,\ldots,6$ we can choose
  $m_j$ and $\tilde\rho_j$ accordingly to satisfy the compatibility
  conditions.  Thus we obtain examples of $\rG_2$--instantons on each
  of the desingularisations of $Y_0$ by appealing to
  \fullref{thm:a}.

  Note that any choice of resolution data for $T^7/(\Gamma \times
  A)$ lifts to an $A$--invariant choice of resolution data for
  $T^7/\Gamma$.  We can then carry out Joyce's generalised Kummer
  construction in a $A$--invariant way and lift up the
  $\rG_2$--instanton constructed above.  However, we could not have
  constructed this $\rG_2$--instanton directly using
  \fullref{thm:a}, since the lift of $\theta$ to $T^7/\Gamma$ is
  not acyclic.
\end{example}

\begin{example}
  Here is a more complicated example.  Let $Y_0:=T^7/(\Gamma \times
  A)$ be as before.  Define $\rho\co \pi_1(Y_0)\to{}V\subset\SO(3)$ by
  \begin{alignat*}{5}
    \gamma,\tau_1,\ldots,\tau_7 &\mapsto 1, &\qquad
    \alpha &\mapsto a, &\qquad
    \beta &\mapsto b, \\
    \sigma_2 &\mapsto b, & \qquad
    \sigma_3 &\mapsto a. &&
  \end{alignat*}
  Again, the resulting flat connection $\theta$ is acyclic.  For
  $j=1,2,3$ let $A_j:=A_{0,1}$ be the rigid ASD instanton on
  $E_j:=E_{0,1}$.  By adapting the framings of $E_2$ and $E_3$, we can
  arrange that $A_2$ and $A_3$ are asymptotic at infinity to the flat
  connection with monodromy given by $b\in V$. For $j=4,\ldots,7$ let
  $A_j$ be the product connection on the trivial bundle $E_j$.  To be
  able to extend this to compatible gluing data we need a lift
  $\tilde\rho_j$ of the action of $\Z_2$ on $X_j$ to $E_j$ preserving
  $A_j$ and acting trivially on the framing at infinity for $j=2,3$.
  If $X_j$ is a smoothing of $\C^2/\Z_2$, then the $\Z_2$ action on
  $X_j$ does lift to $E_j$ preserving $A_j$.  However, the action does
  not lift if $X_j$ is the resolution of $\C^2/\Z_2$.  The reason for
  this is that in the first case the action of $\Z_2$ on $H^2(X,\R)$
  is given by the identity, while in the second case it acts via
  multiplication by~$-1$; see Joyce~\cite[pages~313--314]{Joyce2000}.  Thus
  we can only find compatible gluing data if we resolve both $S_2$ and
  $S_3$ using a smoothing of $\C^2/\Z_2$.

  Here is a small modification of this example.  Define
  $\rho\co\pi_1(Y_0)\to{}V\subset\SO(3)$ by
  \begin{alignat*}{5}
    \gamma,\tau_1,\ldots,\tau_7 &\mapsto 1, &\qquad
    \alpha &\mapsto a, &\qquad
    \beta &\mapsto b, \\
    \sigma_2 &\mapsto b, & \qquad
    \sigma_3 &\mapsto c.&&
  \end{alignat*}
  To find compatible gluing data, one simply has to compose
  $\tilde\rho_j$ as above with multiplication by $b \in
  \sG(E_j)$, for $j=2,3$.
\end{example}

\begin{example}
  Let $B:=\<\sigma_1,\sigma_2,\sigma_3\>$ and $Y_0:=T^7/(\Gamma
  \times B)$.  Then the singular set of $Y_0$ consists of $4$ copies
  of $T^3/\Z_2$, denoted by $S_1,\ldots,S_4$, each of which has a
  neighbourhood modelled on $(T^3\times\C^2/\Z_2)/\Z_2$ where $\Z_2$
  acts on $\C^2/\Z_2$ by $\pm(z_1,z_2)\mapsto\pm(z_1,-z_2)$.  The
  orbifold fundamental group $\pi_1(Y_0)$ is given by
  \begin{equation*}
    \pi_1(Y_0)=\<\alpha,\beta,\gamma,\sigma_1,\sigma_2,\sigma_3,\tau_1,\ldots,\tau_7\>\subset\Aff(7).
  \end{equation*}
  Up to conjugation the fundamental groups of the neighbourhoods $T_j$
  of $S_j$ are given by
   \begin{align*}
    \pi_1(T_1)&=\<\alpha,\tau_4^{-1}\tau_5^{-1}\beta\sigma_1\sigma_2\sigma_3,\tau_1,\tau_2,\tau_3\>, &
    \pi_1(T_2)&=\<\beta,\sigma_3\alpha,\tau_1,\tau_4,\tau_5\>, \\
    \pi_1(T_3)&=\<\gamma,\alpha\beta,\sigma_2,\tau_4,\tau_6\>, &
    \pi_1(T_4)&=\<\tau_3\gamma,\tau_3\alpha\beta,\sigma_2,\tau_4,\tau_6\>.
  \end{align*}
   Define
  $\rho\co \pi_1(Y_0)\to V\subset\SO(3)$ by
  \begin{alignat*}{3}
    \alpha,\beta,\sigma_3,\tau_1,\ldots,\tau_7 &\mapsto 1, &\qquad
    \gamma &\mapsto b \\
    \sigma_1 &\mapsto a, &\qquad
    \sigma_2 &\mapsto b. 
  \end{alignat*}
  The induced flat connection $\theta$ is clearly acyclic.  As
  before, for $j=3,4$, we require $S_j$ to be desingularised using a
  resolution of $\C^2/\Z_2$ in order to be able to find a lift $\tilde
  \rho_j$.  Also note that, for $j=3,4$, now we have make to a
  non-trivial choice for $m_j$, but this causes no problem since $b\in
  V$ lies in $\sG(E_j)$ and preserves $A_j$.

  Again, the resulting $\rG_2$--instanton can be lifted to appropriate
  $\sigma_1$--invariant desingularisations of $T^7/(\Gamma\times
  A)$; however we could not have constructed the lifted
  $\rG_2$--instanton directly, since the lift of $\theta$ to
  $T^7/(\Gamma \times A)$ it is not acyclic.
\end{example}

This list of examples is not exhaustive.  The reader will have no
difficulty finding more examples by modifying the ones given
above.

\hypertarget{appendixa}{}
\section*{Appendix: 
An infinite-dimensional Liouville-type theorem}
\setobjecttype{App}
\def\thesection{A}
\setcounter{equation}{0}
\setcounter{subsection}{0}
\setcounter{theorem}{0}
\label{app:liouville}

The following result is an abstraction of various results that have
appeared in the literature, for example, in Pacard--Ritor\'e's work on
the Allen--Cahn equation~\cite[Corollary~7.5]{Pacard2003} and in
Brendle's unpublished work on the Yang--Mills equation in higher
dimension~\cite[Proposition~3.3]{Brendle2003}.

\begin{lemma}\label{lem:liouvilleA}
  Let $E$ be a vector bundle of bounded geometry over a Riemannian
  manifold $X$ of bounded geometry and with subexponential volume
  growth, and suppose that $D\co C^\infty(X,E)\to{}C^\infty(X,E)$ is a
  uniformly elliptic operator of second order whose coefficients and
  their first derivatives are uniformly bounded, that is non-negative,
  such that $\<Da,a\>\geq 0$ for all $a \in W^{2,2}(X,E)$, and formally
  self-adjoint.  If $a \in C^\infty(\R^n\times X,E)$ satisfies
  \begin{equation*}
    (\Delta_{\R^n}+D)a=0
  \end{equation*}
  and $\|a\|_{L^\infty}$ is finite, then $a$ is constant in the
  $\R^n$--direction, that is $a(x,y)=a(y)$.  Here, by slight abuse of
  notation, we denote the pullback of $E$ to $\R^n\times X$ by $E$ as
  well.
\end{lemma}

Here is a heuristic argument.  Denote by $\hat a$ the partial Fourier
transform of $a$ in the $\R^n$--direction.  Then $\hat a$ solves
$(D+|k|^2)\hat a=0$.  But $D+|k|^2$ is invertible for $k\neq 0$.
Thus $\hat a$ is supported on $\{0\}\times X$ and hence must be a
linear combination of derivatives of various orders of
$\Gamma(E)$--valued $\delta$--functions.  Reversing the Fourier
transform shows that $a$ must be a polynomial in $\R^n$.  But then it
follows from the assumptions that $a$ is constant in the
$\R^n$--direction.  The actual proof will be slightly more pedestrian.

First we need to set-up some notation.  We fix a point $p\in X$ and
denote by $\rho\co X\to[0,\infty)$ a smoothing of the distance from
$p$, as in Kordyukov~\cite[Proposition~4.1]{Kordyukov1991}.  For $\delta\in\R$
we introduce a weight function $w_\delta:=e^{-\delta\rho}$ and
weighted Hilbert spaces $W^{s,2}_\delta(X,E)$ consisting of locally
integrable sections $f$ such that $w_\delta\cdot f$ lies in
$W^{s,2}(X,E)$ with inner product defined by
$\<\,\cdot\,,\,\cdot\,\>_{\smash{W^{s,2}_\delta}}
:=\<w_\delta\cdot,w_\delta\cdot\>_{\smash{W^{s,2}}}$.
As usual we set $L^2_\delta(X,E):=W^{0,2}_\delta(X,E)$.

\begin{prop}\label{prop:dk}
  For each $k_0>0$ there is a constant $\epsilon=\epsilon(k_0)>0$ such
  that for all $\delta\in(-\epsilon,\epsilon)$ and 
  $k\in[k_0,\infty)$ the operator $D+k^2\co W^{2,2}_\delta(X,E)\to
  L^2_\delta(X,E)$ is an isomorphism.  Moreover, for $\ell\geq 0$
  there is a constant $c_\ell=c_\ell(k_0)>0$ such that
  \begin{equation}\label{eq:dkest}
    \big\|\del_k^\ell\big(D+k^2\big)^{-1}a\big\|_{W^{2,2}_\delta}\leq c_\ell(1+k)^\ell \|a\|_{L^2_\delta}
  \end{equation}
  for all $k\in[k_0,\infty)$ and $a\in L^2_\delta(X,E)$.
\end{prop}

\begin{proof}
  By standard elliptic theory we have
  \begin{equation*}
    \|a\|_{W^{2,2}}\leq c \big(\|Da\|_{L^2} + \|a\|_{L^2}\big).
  \end{equation*}
  Since $D$ is non-negative, we have
  \begin{equation*}
    \|Da\|_{L^2}\leq\big\|\big(D+k^2\big)a\big\|_{L^2} \qandq
     k^2\|a\|_{L^2} \leq \big\|\big(D+k^2\big)a\big\|_{L^2}.
  \end{equation*}
  Putting everything together yields
  \begin{equation*}
    \|a\|_{W^{2,2}}\leq c(1+1/k_0^2) \big\|\big(D+k^2\big)a\big\|_{L^2}
  \end{equation*}
  for $k\in[k_0,\infty)$.  This implies that $D+k^2\co W^{2,2}\to L^2$
  is an injective operator with closed range.  It is also surjective,
  since its co-kernel can be identified with the $L^2$ kernel of
  $D+k^2$ which is trivial. 

  We now argue as in~\cite[Proposition~4.4]{Kordyukov1991}.  Via the
  Hilbert space isomorphism $W^{s,2}_{\smash{\delta}} \iso W^{s,2}$ defined by
  multiplication with $w_\delta$ the operator $D+k^2\co
  W^{2,2}_{\smash{\delta}}\to \smash{L^2_\delta}$ is equivalent to $D_\delta+k^2\co
  W^{2,2}\to L^2$ where $D_\delta:=w_\delta D w_{\smash{\delta}}^{-1}$.  We can
  write $D_\delta$ as
  \begin{equation*}
    D_\delta=D+\delta P_\delta
  \end{equation*}
  with $P_\delta\co W^{2,2}\to L^2$ bounded independent of
  $\delta$.  Therefore,
  \begin{equation*}
    \big\|\big(\big(D+k^2\big)-\big(D_\delta+k^2\big)\big)\big(D+k^2\big)^{-1} a\big\|_{L^2}
    \leq |\delta| c(1+1/k_0^2) \|a\|_{L^2}.
  \end{equation*}
  If we choose $\epsilon=\epsilon(k_0)>0$ sufficiently small, then for
  $\delta\in(-\epsilon,\epsilon)$ the factor on the right-hand sight
  is less than $\frac12$; thus, the series
  \begin{equation*}
    (D_\delta+k^2)^{-1} 
    := \big(D+k^2\big)^{-1} \sum_{i\geq 0}
\big[\big(\big(D+k^2\big)-\big(D_\delta+k^2\big)\big)\big(D+k^2\big)^{-1}\big]^i
  \end{equation*}
  converges and the operator norm of $(D_\delta+k^2)^{-1}$ is
  bounded by $2c(1+1/k_0^2)$.  This establishes \eqref{eq:dkest} for
  $\ell=0$.  For $\ell>0$, we have
  \begin{equation*}
    \del_k^\ell\big(D+k^2\big)^{-1}=
    \sum_{i=0}^{\ell}\sum_{j=2}^{\ell+1} c_{i,j,\ell} \cdot k^i
\big[\big(D+k^2\big)^{-1}\big]^j
  \end{equation*}
  for universal constants $c_{i,j,\ell}$.  Thus \eqref{eq:dkest} for
  $\ell>0$ can be reduced to the case $\ell=0$.
\end{proof}

\fullref{lem:liouvilleA} can now be proved using an argument similar
to the one used by Brendle in~\cite[Proposition~3.3]{Brendle2003}.
This is essentially the proof of the ingredients from classical
distribution theory used in the heuristic proof adapted to our
infinite-dimensional setting.

\begin{proof}[Proof of \fullref{lem:liouvilleA}]
  We proceed in $3$ steps.
  \setcounter{step}{0}
  \begin{step}\label{step:l1}
    Let $\chi\in\sS(\R^n)$ be a fast decaying function whose Fourier
    transform $\hat\chi$ vanishes in $B_{\smash{k_0}}(0)$ and let $b\in
    \smash{L^2_\delta}(X,E)$ for some $\delta\in(-\epsilon,\epsilon)$ with
    $\epsilon=\epsilon(k_0)$.  Then there exists $a \in
    \sS(\R^n,W^{2,2}_{\smash{\delta}}(X,E))$ such that
    $(\Delta_{\R^n}+D) a = \chi b$.
  \end{step}

  We construct $a\in\sS(\R^n,W^{2,2}_\delta(X,E))$ using Fourier
  synthesis.  By assumption 
  $\hat\chi(k)=0$ for $|k|\leq k_0$.  For $|k|>k_0$ set
  \begin{equation*}
    \hat a_k:=\big(D+|k|^2\big)^{-1}b.
  \end{equation*}
  and define
  \begin{equation*}
    a(x,y):=\int_{\R^n} e^{i\<x,k\>} \hat a_k(y) \hat\chi(k)
    ~\rd\cL^n(k).
  \end{equation*}
  Here $\cL^n$ denotes the $n$--dimensional Lebesgue measure on $\R^n$.
  Then
  \begin{equation*}
    \big(\Delta_{\R^n}+D\big)a(x,y) =b\chi.
  \end{equation*}
  Moreover, one can verify that $x
\mapsto\|a(x,\cdot\,)\|_{\smash{W^{2,2}_\delta}}$ is
  in $\sS(\R^n)$ using a slight variation of the proof that the
  Fourier transform maps fast decaying functions to fast decaying
  functions and the estimate $\smash{\|\del_k^\ell \hat
a_k\|_{W^{2,2}_\delta}}
  \leq c_\ell (1+|k|)^\ell \smash{\|b\|_{L^2_\delta}}$.

  \begin{step}\label{step:l2}
    Let $\chi\in\sS(\R^n)$ with $\hat\chi(0)=0$.  Then there is a
    family $(\chi_\epsilon)_{\epsilon>0}$ of fast decaying functions
    such that $\hat\chi_\epsilon$ vanishes on $B_\epsilon(0)$ and
    $\lim_{\epsilon\to 0} \|\chi_\epsilon-\chi\|_{L^1}=0$.
  \end{step}

  Pick a smooth function $\rho \co \R \to [0,1]$ such that $\rho(k) =
  0$ for $|k|\leq 1$ and $\rho(k)=1$ for $|k|\geq 2$.  Set
  $\hat\chi_{\epsilon}(k):=\rho(|k|/\epsilon)\hat\chi(k)$ and denote
  its inverse Fourier transform by $\chi_\epsilon$.  Then
  $\chi_\epsilon$ clearly satisfies the first part of the conclusion.
  To see that the second part also holds, note that from
  $\hat\chi(0)=0$ it follows that
  \begin{equation*}
    \big\|\nabla^n(\hat\chi_\epsilon-\hat\chi)\big\|_{L^{2n/(2n-1)}}
      =O\big(\epsilon^{\unfrac12}\big)
  \end{equation*}
  and therefore
  \begin{align*}
    \|\chi_\epsilon-\chi\|_{L^1}
     &\leq
\big\|(1+|x|)^{-n}\big\|_{L^{2n/(2n-1)}}\cdot\big\|(1+|x|)^n(\chi_\epsilon-\chi)\big\|_{L^{2n}} \\
     &\leq c
\big(\|\hat\chi_\epsilon-\hat\chi\|_{L^{2n/(2n-1)}}+\big\|\nabla^n(\hat\chi_\epsilon-\hat\chi)\big\|_{L^{2n/(2n-1)}}\big)
      =O\big(\epsilon^{\unfrac12}\big),
  \end{align*}
  where $c>0$ is a constant depending only on $n$.  Here we used that
  the inverse Fourier transform is a bounded linear map from
  $L^{2n/\(2n-1\)}$ to $L^{2n}$ and the Fourier transform's behaviour
  with respect to derivatives.

  \begin{step}
    Suppose that $(\Delta_{\R^n}+D)a=0$.  Then for
    $\sigma\in\sS^n(\R^n)$, $\delta\in\R^n$ and $b\in C^\infty_c(X,E)$
    we have
    \begin{equation*}
      \int_{\R^n}
      \<a(x,\cdot\,),b\>_{L^\infty,L^1}(\sigma(x+\delta)-\sigma(x))~\rd\cL^n(x)=0.
    \end{equation*}
    In particular, the conclusion of the lemma holds.
  \end{step}

  Set $\chi(x):=\sigma(x+\delta)-\sigma(x)$.  Then $\hat\chi(0)=0$.
  Let $\chi_\epsilon$ be as in Step~\ref{step:l2}.  According to
  Step~\ref{step:l1}, for each $\epsilon>0$ there is some small
  $\delta>0$ and $c_\epsilon\in \sS(\R^n,W^{2,2}_{\smash{-\delta}}(X,E))$
  such that $(\Delta_{\R^n}+D)c_\epsilon = \chi_\epsilon b$.  By the
  assumptions on $a$ and since $X$ has subexponential volume growth we
  have
  \begin{align*}
    \int_{\R^n} \<a(x,\cdot\,),b\> \chi(x)~\rd\cL^n(x)
    &=\lim_{\epsilon\to 0} \int_{\R^n}
      \<a(x,\cdot\,),b\> \chi_\epsilon(x)~\rd\cL^n(x) \\
    &=\lim_{\epsilon\to 0} \int_{\R^n} \int_X
      \<a(x,y),(\Delta_{\R^n}+D)c_\epsilon\> ~\rd\cL^n(x)~\rd\vol(y) \\
    &=\lim_{\epsilon\to 0} \int_{\R^n} \int_X \<(\Delta_{\R^n}+D) a(x,y),c_\epsilon\>
      ~\rd\cL^n(x)~\rd\vol(y)\\&=0.
  \end{align*}
  Since $\sigma$, $\delta$ and $b$ are arbitrary, it follows that $a$
  is invariant in the $\R^n$--direction.  This finishes the proof.
\end{proof}

\begin{remark}
  It is clear from the proof that in \fullref{lem:liouville} one can
  replace the assumptions that $X$ has subexponential volume growth
  and that $\|a\|_{L^\infty}$ is finite by the assumption that
  $\|a(x,\cdot)\|_{L^2_\delta}$ is bounded independent of $x\in\R^n$
  for all $\delta>0$.
\end{remark}

\bibliographystyle{gtart}
\bibliography{link}

\end{document}